\documentclass[a4paper,12pt]{article}

\usepackage{amsmath}
\usepackage{ntheorem}
\usepackage{amssymb}
\usepackage{latexsym}
\usepackage{mathrsfs}
\usepackage[all]{xy}\UseComputerModernTips
\usepackage{graphicx}

\theoremstyle{plain}
\newtheorem{prp}{Proposition}[section]
\newtheorem{thm}[prp]{Theorem}
\newtheorem{lem}[prp]{Lemma}
\newtheorem{cor}[prp]{Corollary}

\theoremstyle{plain}
\theorembodyfont{\normalfont\rm}
\newtheorem{dfn}[prp]{Definition}
\newtheorem{rem}[prp]{Remark}
\newtheorem{exm}[prp]{Example}

\theoremstyle{nonumberplain}
\theoremheaderfont{\normalfont\itshape}
\theoremseparator{.}
\theorembodyfont{\normalfont}
\newtheorem{proof}{Proof}

\newcommand{\qed}{\hfill $\Box$}

\makeatletter
\renewcommand{\theequation}{%
\thesection.\arabic{equation}}
\@addtoreset{equation}{section}
\makeatother

\newcommand{\Real}{{\mathbb R}}
\newcommand{\Comp}{{\mathbb C}}
\newcommand{\Z}{{\mathbb Z}}

\newcommand{\D}{{\rm D}}
\newcommand{\Db}{{\bf D}^{b}}
\newcommand{\Dc}{{\bf D}_{\Real -c}^{b}}

\newcommand{\Lb}{{\cal F}}

\newcommand{\Vb}{{\cal G}}

\newcommand{\MW}{{\cal W}}
\newcommand{\ME}{{\cal E}}
\newcommand{\MS}{{\cal S}}

\renewcommand{\L}{{\mathscr L}}

\renewcommand{\1}{{\bf 1}}
\newcommand{\CF}{{\rm CF}}

\newcommand{\reg}{{\rm reg}}
\newcommand{\tr}{{\rm tr}}
\newcommand{\id}{{\rm id}}
\newcommand{\supp}{{\rm supp}}
\renewcommand{\SS}{{\rm SS}}
\newcommand{\BM}{{\rm BM}}
\newcommand{\Ev}{{\rm Ev}}
\newcommand{\sgn}{{\rm sgn}}
\newcommand{\RG}{R\varGamma}

\newcommand{\Hom}{{\cal H}om}
\newcommand{\e}{\varepsilon}

\renewcommand{\t}[1]{{^{t}{#1}^{\prime}}}
\newcommand{\tl}[1]{\widetilde{#1}}

\newcommand{\utimes}[1]{\times_{#1}}
\renewcommand{\d}{{\rm dim}}

\newcommand{\simto}{\overset{\sim}{\longrightarrow}}
\newcommand{\simot}{\overset{\sim}{\longleftarrow}}
\newcommand{\dsum}{\displaystyle \sum}
\newcommand{\dint}{\displaystyle \int}

\newcommand{\inun}{\text{\rotatebox[origin=c]{90}{$\in$}}}
\newcommand{\longhookrightarrow}{\DOTSB\lhook\joinrel\longrightarrow}
\newcommand{\longtwoheadrightarrow}{\relbar\joinrel\twoheadrightarrow}

\let\Re=\relax
\let\Im=\relax

\DeclareMathOperator{\Re}{\mathrm{Re}}
\DeclareMathOperator{\Im}{\mathrm{Im}}

\title{Hyperbolic localization and Lefschetz fixed point formulas for
higher-dimensional fixed point sets}

\author{
Yuichi IKE\footnote{Graduate School of Mathematical Sciences, the University
of Tokyo, 3-8-1, Komaba, Meguro-ku, Tokyo, 153-8914, Japan, E-mail:
ike@ms.u-tokyo.ac.jp} \and
Yutaka MATSUI\footnote{Department of Mathematics, Kindai University, 3-4-1,
Kowakae, Higashi-Osaka, Osaka, 577-8502, Japan, E-mail:
matsui@math.kindai.ac.jp} \and
Kiyoshi TAKEUCHI\footnote{Institute of Mathematics, University  of Tsukuba,
1-1-1, Tennodai, Tsukuba, Ibaraki, 305-8571, Japan, E-mail:
takemicro@nifty.com} }
\date{ }

\begin{document}

\maketitle

\begin{abstract}
We study Lefschetz fixed point formulas for constructible sheaves with
higher-dimensional fixed point sets. Under fairly weak assumptions, we prove
that the local contributions from them are expressed by some constructible
functions associated to hyperbolic localizations. 
This gives an affirmative answer to a conjecture 
of Goresky-MacPherson \cite{G-M-2} 
in particular for smooth fixed point components 
(see \cite[page 9, (1.12) Open problems]{G-M-1}). 
In the course of the 
proof, the new Lagrangian cycles introduced in our previous paper
\cite{M-T-3} will be effectively used. Moreover we show various examples for
which local contributions can be explicitly determined by our method.
\end{abstract}

\section{Introduction}
Lefschetz fixed point formulas are important in many branches of mathematics
such as topology, algebraic geometry, number theory, dynamical systems and
representation theory. Despite a lot of activities on this subject, the case
where the fixed point set is higher-dimensional still remains quite
mysterious. In this paper we study Lefschetz fixed point formulas for
morphisms $\phi \colon X \longrightarrow X$ of real analytic manifolds $X$
whose fixed point set $M= \{ x \in X \ |\ \phi(x)=x\} \subset X$ is
higher-dimensional (since we mainly consider the case where the fixed point
set is a smooth submanifold of $X$, we use the symbol $M$ to express it). It
is well-known that when $X$ is compact the global Lefschetz number of $\phi$
\begin{equation}
\tr(\phi):=\dsum_{j \in \Z} (-1)^j
\tr\{H^j(X;\Comp_X)\overset{\phi^*}{\longrightarrow} H^j(X;\Comp_X) \} \in
\Comp
\end{equation}
is expressed as the integral of a local cohomology class $C(\phi) \in
H_M^n(X;or_X)$ supported by $M$, where we set $\d X=n$ and $or_X$ is the
orientation sheaf of $X$ (see Dold \cite{Dold}, \cite{Dold-2} etc.). Let $M=
\bigsqcup_{i \in I}M_i$ be the decomposition of $M$ into connected
components and
\begin{gather}
H_M^n(X;or_X) =\bigoplus_{i \in I} H_{M_i}^n (X;or_X),\\
C(\phi) =\bigoplus_{i \in I} C(\phi )_{M_i}
\end{gather}
the associated direct sum decompositions. We call the integral
$c(\phi)_{M_i} \in \Comp$ of the local cohomology class $C(\phi)_{M_i} \in
H_{M_i}^n (X;or_X)$ for a fixed point component $M_i$ the local contribution
from $M_i$. Then the global Lefschetz number of $\phi$ is equal to the sum
of the local contributions from $M_i$'s:
\begin{equation}
\tr(\phi) =\dsum_{i \in I} c(\phi)_{M_i}.
\end{equation}
But if the fixed point component $M_i$ is higher-dimensional, it is in
general very difficult to compute the local contribution by the following
dimensional reason. Let $M_i$ be a fixed point component of $\phi$ whose
codimension $d$ in $X$ satisfies the condition $0<d<n$. Then the local
cohomology group $H_{M_i}^n (X;or_X)$ is isomorphic to the $0$-dimensional
Borel-Moore homology group $H^{\BM}_0(M_i;\Comp)$ of $M_i$ by the Alexander
duality, and the class $C(\phi)_{M_i}$ in it cannot be calculated locally at
each point of $M_i$. On the other hand, top-dimensional Borel-Moore homology
cycles in $M_i$, i.e. elements in $H^{\BM}_{n-d}(M_i;\Comp)$
can be much more easily handled since they are realized as sections of a
relative orientation sheaf on $M_i$.

In this paper, we overcome this difficulty partially 
by using our new Lagrangian cycles introduced in \cite{M-T-3}. Since
we also want to study Lefschetz fixed point formulas over singular varieties
(and those for intersection cohomology groups), from now we consider the
following very general setting. Let $X$, $\phi$ and $M= \bigsqcup_{i \in
I}M_i$ be as before, and $F$ a bounded complex of sheaves of
$\Comp_X$-modules whose cohomology sheaves are $\Real$-constructible in the
sense of \cite{K-S}. Assume that we are given a morphism $\Phi \colon
\phi^{-1}F \longrightarrow F$ in the derived category $\Dc(X)$. If the
support $\supp(F)$ of $F$ is compact, we can define the global trace
(Lefschetz number) $\tr (F,\Phi) \in \Comp$ of the pair $(F, \Phi)$ by
\begin{equation}
\tr (F,\Phi) :=\dsum_{j \in \Z} (-1)^j \tr
\{H^j(X;F)\overset{\Phi}{\longrightarrow} H^j(X;F) \} \in \Comp,
\end{equation}
where the morphisms $H^j(X;F) \overset{\Phi}{\longrightarrow} H^j(X;F)$ are
induced by
\begin{equation}
F \longrightarrow R\phi_*\phi^{-1}F \overset{\Phi}{\longrightarrow}R\phi_*F.
\end{equation}
In this very general setting, Kashiwara \cite{Kashiwara-2} introduced local
contributions $c(F,\Phi)_{M_i} \in \Comp$ from the fixed point components
$M_i$ and proved the equality
\begin{equation}
\tr (F,\Phi)=\dsum_{i \in I} c(F,\Phi)_{M_i}.
\end{equation}
Therefore the remaining task for us is to calculate the local contributions
$c(F,\Phi)_{M_i}\in \Comp$ explicitly. Let $M_i$ be 
a fixed point component of $\phi$ 
whose regular part $(M_i)_{\reg} \subset M_i$ 
satisfies the condition $\supp (F) \cap M_i \subset 
(M_i)_{\reg}$. For the sake of simplicity, 
we denote $(M_i)_{\reg}$ simply by $M$. Then
there exists a natural morphism
\begin{equation}
\phi^{\prime} \colon T_MX \longrightarrow T_MX
\end{equation}
induced by $\phi$, where $T_MX$ is the normal bundle of $M$ in $X$. For each
point $x\in M$, we define a finite subset $\Ev(\phi^{\prime}_x)$ of $\Comp$ by
\begin{equation}
\Ev(\phi^{\prime}_x):=\{ \text{the eigenvalues of $\phi^{\prime}_x \colon (T_{
M}X)_x
\longrightarrow (T_MX)_x$} \} \subset \Comp.
\end{equation}
Assume the condition:
\begin{equation}\label{eq:1-9}
1 \notin \Ev(\phi^{\prime}_x) \hspace{5mm} \text{for any $x \in \supp(F) \cap 
M$},
\end{equation}
which means that the graph $\Gamma_{\phi}=\{ (\phi(x), x) \ | \ x\in X\}
\subset X \times X$ of $\phi$ intersects cleanly (see \cite[Definition
4.1.5]{K-S}) the diagonal set $\Delta_X \subset X \times X$ along $M \subset
\Gamma_{\phi} \cap \Delta_X$ on $\supp(F) \cap M$. It naturally appears also
in the classical study of Atiyah-Bott type Lefschetz theorems by 
Gilkey \cite[Theorem 3.9.2]{Gilkey}, Lee \cite{Lee} and Toledo-Tong \cite{T-T-3}.
Under the condition \eqref{eq:1-9}, in
\cite{M-T-3} we constructed a new Lagrangian cycle $LC(F, \Phi)_M$ in the
cotangent bundle $T^*M$. We call it the Lefschetz cycle associated with the
pair $(F, \Phi)$ and the fixed point component $M$. Note that in the more
general setting of elliptic pairs a similar construction of microlocal
Lefschetz classes was previously given
in Guillermou \cite{Guillermou}. 
The difference from his is that we explicitly realized them
as Lagrangian cycles in $T^*M$.
For recent results on this subject, see also
\cite{Ike}, \cite{K-S-2} and \cite{R-T-T} etc.
Note also that if $\phi=\id_X$, $M=X$ and
$\Phi =\id_F$, our Lefschetz cycle $LC(F, \Phi)_M$ coincides with the
characteristic cycle $CC(F)$ of $F$ introduced by Kashiwara
\cite{Kashiwara-1}. By
Lefschetz cycles, in \cite{M-T-3} we could generalize almost all nice
properties of characteristic cycles. In particular, we have the following
microlocal index theorem for the local contribution $c(F, \Phi)_M$ from $M$.

\begin{thm}{\rm \bf (\cite[Theorem 4.8]{M-T-3})}
Assume that $\supp (F) \cap M$ is compact. Then for any continuous section
$\sigma \colon M \longrightarrow T^*M$ of $T^*M$, we have
\begin{equation}
c(F,\Phi)_M =\sharp ([\sigma] \cap LC(F,\Phi)_M),
\end{equation}
where $\sharp ([\sigma] \cap LC(F,\Phi)_M)$ is the intersection number of
the image of $\sigma$ and $LC(F,\Phi)_M$ in the cotangent bundle $T^*M$.
\end{thm}

However in \cite{M-T-3} we could not describe $LC(F,\Phi)_M$ explicitly in
terms of $(F, \Phi)$ and $M$ since it was defined merely in an abstract
manner by some morphisms in derived categories. Our answer to this problem
is as follows. Let
\begin{equation}
CC \colon \CF(M)_{\Comp} \simto \varGamma(T^*M;\L_M)
\end{equation}
be the isomorphism between the $\Comp$-vector space consisting of the
$\Comp$-valued (subanalytically) constructible functions on $M$ and that of
closed conic subanalytic Lagrangian cycles on $T^*M$ with coefficients in
$\Comp$ (see Proposition \ref{prp:2-10}). Then in Section \ref{sec:6} we
define a $\Comp$-valued constructible function $\theta (F,\Phi)_M \in
\CF(M)_{\Comp}$ on $M$ associated to the hyperbolic localization of the
specialization $\nu_M(F)$ of $F$ in the sense of Braden \cite{Braden},
Braden-Proudfoot \cite{B-P} and Goresky-MacPherson \cite{G-M-1}. 
More precisely, for each 
point $x \in M$ by taking an expanding subbundle 
(see Definition \ref{dfn:6-2}) 
$\ME \subset \Vb |_U$ of $\Vb =T_MX$ on its 
neighborhood $U \subset M$ we set 
\begin{equation}
\theta (F,\Phi)_M 
(x):=\dsum_{j \in \Z} (-1)^j\tr\{ H^j (\nu_M(F)_\ME^{!-1})_x
\overset{(\Phi^{\prime})_\ME^{!-1}
|_{\{x\}}}{\longrightarrow} H^j(\nu_M(F)_\ME^{!-1}
)_x\}, 
\end{equation}
where $\nu_M(F)_\ME^{!-1} \in \Dc(U)$ is the 
hyperbolic localization of $\nu_M(F)$ with 
respect to $\ME$ and $(\Phi^{\prime})_\ME^{!-1}: 
\nu_M(F)_\ME^{!-1} \longrightarrow \nu_M(F)_\ME^{!-1}$ 
is its endomorphism induced by $\Phi$. 
Then we have the following result. 

\begin{thm}\label{M-THM}
We have the equality
\begin{equation}
LC(F,\Phi)_M=CC( \theta (F,\Phi)_M )
\end{equation}
as Lagrangian cycles in $T^*M$. 
In particular, if moreover 
$\supp(F) \cap M$ is compact we have
\begin{equation}
c(F,\Phi)_M =  \dint_M \theta (F,\Phi)_M,
\end{equation}
where $\dint_M \colon \CF(M)_{\Comp} \longrightarrow \Comp$ is the morphism
defined by topological (Euler) integrals (see Definition \ref{dfn:2-9}).
\end{thm}

Note that when $\d M =0$ this theorem coincides with
Kashiwara's one in \cite[Proposition 1.4.1]{Kashiwara-2} (see also
Kashiwara-Schapira \cite[Proposition 9.6.12]{K-S}). 
It gives also an affirmative answer to a conjecture 
of Goresky-MacPherson \cite{G-M-2} 
in particular for smooth fixed point components 
(see also \cite[page 9, (1.12) Open problems]{G-M-1}). 
In the previous results in \cite{G-M-1}, 
the authors assumed a technical condition 
that there exists a special indicator map 
$t: W \longrightarrow \Real_{\geq 0} \times 
\Real_{\geq 0}$ on a neighborhood $W$ of 
$M$ in $X$ (see \cite[Definition (3.1)]{G-M-1}). 
They call such an endomorphism $\phi$ 
a weakly hyperbolic map. Moreover by taking $\ME$ 
to be the minimal expanding subbundle 
(see Definition \ref{dfn:6-1}) $\MW \subset \Vb|_U$ 
of $\Vb =T_MX$ on $U \subset M$ we can 
reduce the calculation of 
the local contribution $c(F,\Phi)_M$ to that on a 
subset $\MW \subset \Vb |_U$ much smaller than the 
one in \cite{G-M-1}. In contrast to loc.\ cit., 
at each point of $M$ 
we can take the smallest possible $\ME$ 
containing only the 
generalized eigenspaces for the eigenvalues 
of $\phi^{\prime}_x$ in $\Real_{> 1}$. 
Compare our stalk formula of $\nu_M(F)_\ME^{!-1}$ 
in Proposition \ref{prp:6-7} with the theorem in 
\cite[page 6]{G-M-1}. 
Thus Theorem \ref{M-THM} improves the results in 
\cite{Kashiwara-2}, \cite{K-S} and \cite{G-M-1} etc.
In order to prove it even in the difficult
case where the set $\Ev(\phi^{\prime}_x)$ varies
depending on $x \in \supp(F) \cap M$,
we require some precise arguments on
our Lefschetz cycles (see Remark \ref{rem:3-6}). 
More precisely, we encode Kashiwara's characteristic class 
$C(F,\Phi)_M \in H_{\supp(F) \cap M}^n(X;or_X)$ satisfying  
$\int_X C(F,\Phi)_M = c(F,\Phi)_M$ 
to the geometric cycle $LC(F,\Phi)_M$ in $T^*M$ 
and calculate the latter locally by applying our 
microlocal index theorem to hyperbolic localizations. 
Namely $LC(F,\Phi)_M$ enables us to 
patch local calculations to a global one. 
See Section \ref{sec:6} for the details. 
Finally in Section \ref{sec:7} we will show
various examples for which the 
function $\theta (F,\Phi)_M$ and the local
contribution $c(F, \Phi)_M \in \Comp$ 
can be explicitly calculated. In
particular, we will give a very 
short proof to the following result proved
first in \cite{M-T-3} by using some deep results on the functorial
properties of Lefschetz cycles (see \cite[Sections 5 and 6]{M-T-3}). 
Let $M= \sqcup_{\alpha \in A} M_{\alpha}$ be the 
decomposition of  
$M=(M_i)_{\reg}$ into its connected components. 
Denote the sign of the determinant of the linear map
\begin{equation}
\id- \phi^{\prime} \colon T_{M_{\alpha}}X \longrightarrow 
T_{M_{\alpha}}X
\end{equation}
simply by $\sgn(\id-\phi^{\prime})_{M_{\alpha}} \in \{ \pm 1 \}$.

\begin{thm}\label{THM-NC}{\rm \bf (\cite[Corollary 6.5]{M-T-3})}
In addition to the condition \eqref{eq:1-9}, 
assume that the inclusion map
$i_M \colon M \longhookrightarrow X$ is 
non-characteristic for $F$ and 
$\supp(F) \cap M$ is compact. Then we have
\begin{equation}
c(F,\Phi)_M = \sum_{\alpha \in A} \ 
\sgn(\id-\phi^{\prime})_{M_{\alpha}} \cdot 
\dint_{M_{\alpha}} \varphi(F|_{M_{\alpha}},
 \Phi|_{M_{\alpha}}),
\end{equation}
where the $\Comp$-valued constructible function
$\varphi(F|_{M_{\alpha}},\Phi|_{M_{\alpha}}) \in
\CF (M_{\alpha})_{\Comp}$ on $M_{\alpha}$ is defined by
\begin{equation}
\varphi(F|_{M_{\alpha}},\Phi|_{M_{\alpha}})
(x):=\dsum_{j \in \Z} (-1)^j\tr\{ H^j(F)_x
\overset{\Phi|_{\{x\}}}{\longrightarrow} H^j(F)_x\}
\end{equation}
for $x\in M_{\alpha}$.
\end{thm}

Thus Theorems \ref{M-THM} and \ref{THM-NC} show
that we can calculate the
global trace of $(F, \Phi)$ locally at each point of
$M$ without assuming any further technical condition such as
\begin{equation}
\Ev(\phi^{\prime}_x) \cap \Real_{> 1}=\emptyset 
\hspace{5mm} \text{for any $x \in M$} 
\end{equation}
on the map $\phi$. Note that if there exists a point $x\in M$ such that
$\Ev(\phi^{\prime}_x) \cap \Real_{> 1}\neq \emptyset$ the classical
methods (see e.g. \cite[Section 9.6]{K-S}) for 
localizations do not work. Our main result in this paper, that is Theorem
\ref{M-THM} is much more general than Theorem
\ref{THM-NC} since in the former we do not assume that
$i_M \colon M \longhookrightarrow X$ is
non-characteristic for $F$. In particular, it immediately implies
general Lefschetz fixed point formulas for singular subvarieties
$V \subset X$ of $X$ such that $\phi (V) \subset V$ 
by applying it to the special case where $F= \Comp_V$ 
and $\Phi$ is the natural morphism $\phi^{-1} \Comp_V 
= \Comp_{\phi^{-1}(V)} \longrightarrow \Comp_V$. 
Note that a complete answer to this problem is known
only for some special cases (for the case of
normal complex algebraic surfaces, see
Saito \cite{Saito}).

\section{Preliminary notions and results}\label{sec:2}

In this paper, we essentially follow the terminology in
\cite{H-T-T} and \cite{K-S}. For
example, for a topological space $X$, we denote by $\Db(X)$ the derived
category of bounded complexes of sheaves of $\Comp_X$-modules on $X$. From
now on, we shall review basic notions and known results concerning Lefschetz
fixed point formulas. Since we focus our attention on Lefschetz fixed point
formulas for constructible sheaves in this paper, we treat here only real
analytic manifolds and morphisms. Now let $X$ be a real analytic manifold.
We denote by $\Dc(X)$ the full subcategory of $\Db(X)$ consisting of bounded
complexes of sheaves whose cohomology sheaves are $\Real$-constructible (see
\cite[Chapter VIII]{K-S} for the precise definition). Let $\phi \colon X
\longrightarrow X$ be an endomorphism of the real analytic manifold $X$.
Then our initial datum is a pair $(F,\Phi)$ of $F \in \Dc(X)$ and a morphism
$\Phi \colon \phi^{-1}F \longrightarrow F$ in $\Dc(X)$. If the support
$\supp (F)$ of $F$ is compact, $H^j(X;F)$ is a finite-dimensional vector
space over $\Comp$ for any $j \in \Z$ and we can define the following
important number from $(F,\Phi)$.

\begin{dfn}\label{dfn:2-1}
We set
\begin{equation}
\tr (F,\Phi) :=\dsum_{j \in \Z} (-1)^j \tr \{ H^j(X;F)
\overset{\Phi}{\longrightarrow} H^j(X;F) \} \in \Comp,
\end{equation}
where the morphisms $H^j(X;F) \overset{\Phi}{\longrightarrow} H^j(X;F)$ are
induced by
\begin{equation}
F \longrightarrow R\phi_*\phi^{-1}F \overset{\Phi}{\longrightarrow}R\phi_*F.
\end{equation}
We call $\tr (F,\Phi)$ the global trace of the pair $(F,\Phi)$.
\end{dfn}

Now let us set
\begin{equation}
M:= \{ x \in X \ |\ \phi(x)=x\} \subset X.
\end{equation}
This is the fixed point set of $\phi \colon X \longrightarrow X$ in $X$.
Since we mainly consider the case where the fixed point set is a smooth
submanifold of $X$, we use the symbol $M$ to express it. If a compact group
$G$ is acting on $X$ and $\phi$ is the left action of an element of $G$,
then the fixed point set is smooth by Palais's theorem \cite{Palais} (see
\cite{Guillemin} for an excellent survey of this subject). Now let us
consider the diagonal embedding $\delta_X \colon X \longhookrightarrow X
\times X$ of $X$ and the closed embedding $h:=(\phi, \id_X) \colon X
\longhookrightarrow X \times X$ associated with $\phi$. Denote by $\Delta_X$
(resp.\ $\Gamma_{\phi}$) the image of $X$ by $\delta_X$ (resp.\ $h$). Then $M
\simeq \Delta_X \cap \Gamma_{\phi}$ and we obtain a chain of morphisms
\begin{eqnarray}
R\Hom_{\Comp_X}(F,F)
&\simeq & \delta_X^!(F \boxtimes \D F) \\
&\longrightarrow & \RG_{\supp(F) \cap \Delta_X} (h_*h^{-1}(F \boxtimes \D
F))|_{\Delta_X} \\
&\simeq & \RG_{\supp(F) \cap \Delta_X} (h_*(\phi^{-1}F \otimes \D
F))|_{\Delta_X} \\
&\overset{\Phi}{\longrightarrow}& \RG_{\supp(F) \cap \Delta_X}(h_*(F \otimes
\D F))|_{\Delta_X} \\
& \longrightarrow & \RG_{\supp(F) \cap \Delta_X} (h_*\omega_X)|_{\Delta_X}
\\
&\simeq & \RG_{\supp(F) \cap M} (\omega_X),
\end{eqnarray}
where $\omega_X \simeq or_X [\d X] \in \Dc(X)$ is the dualizing complex of
$X$ and \\$\D F =R\Hom_{\Comp_X}(F, \omega_X)$ is the Verdier dual of $F$.
Hence we get a morphism
\begin{equation}\label{eq:2-10}
{\rm Hom}_{\Db(X)}(F,F) \longrightarrow H^0_{\supp(F) \cap M}(X;\omega_X).
\end{equation}

\begin{dfn}[\cite{Kashiwara-2}]\label{dfn:2-2} We denote by $C(F,\Phi)$ the
image of $\id_F$ by the morphism \eqref{eq:2-10} in $H_{\supp(F) \cap
M}^0(X;\omega_X)$ and call it the characteristic class of $(F,\Phi)$.
\end{dfn}

\begin{thm}[\cite{Kashiwara-2}]\label{thm:2-3}
If $\supp (F)$ is compact, then the equality
\begin{equation}
\tr (F,\Phi)=\dint_X C(F,\Phi)
\end{equation}
holds. Here
\begin{equation}
\dint_X \colon H_c^n(X;or_X) \longrightarrow \Comp
\end{equation}
is the morphism induced by the integral of differential $(\d X)$-forms with
compact support.
\end{thm}

Let $M= \bigsqcup_{i \in I}M_i$ be the decomposition of $M$ into connected
components and
\begin{gather}
H_{\supp(F) \cap M}^0(X;\omega_X) =\bigoplus_{i \in I} H_{\supp(F) \cap
M_i}^0(X;\omega_X), \\
C(F,\Phi) =\bigoplus_{i \in I} C(F,\Phi)_{M_i}
\end{gather}
the associated direct sum decomposition.

\begin{dfn}\label{dfn:2-4}
When $\supp(F) \cap M_i$ is compact, we define a complex number
$c(F,\Phi)_{M_i}$ by
\begin{equation}
c(F,\Phi)_{M_i}:=\dint_X C(F,\Phi)_{M_i}
\end{equation}
and call it the local contribution of $(F,\Phi)$ from $M_i$.
\end{dfn}

By Theorem \ref{thm:2-3}, if $\supp (F)$ is compact, the global trace of
$(F,\Phi)$ is the sum of local contributions:
\begin{equation}
\tr (F,\Phi)=\dsum_{i \in I} c(F,\Phi)_{M_i}.
\end{equation}
Hence one of the most important problems in the theory of Lefschetz fixed
point formulas is to explicitly describe these local contributions. However
the direct computation of local contributions is a very difficult task in
general. Instead of local contributions, we usually consider first the
following number $\tr (F|_{M_i},\Phi|_{M_i})$ which is much more easily
computed. Let $M_i$ be a fixed point component such that $\supp (F) \cap
M_i$ is compact.

\begin{dfn}\label{dfn:2-5}
We set
\begin{equation}
\tr (F|_{M_i},\Phi|_{M_i}):= \dsum_{j \in \Z}(-1)^j \tr\{ H^j(M_i;F|_{M_i})
\overset{\Phi|_{M_i}}{\longrightarrow} H^j(M_i;F|_{M_i})\},
\end{equation}
where the morphisms $H^j(M_i;F|_{M_i})\overset{\Phi|_{M_i}}{\longrightarrow}
H^j(M_i;F|_{M_i})$ are induced by the restriction
\begin{equation}
\Phi|_{M_i} \colon F|_{M_i} \simeq (\phi^{-1}F)|_{M_i} \longrightarrow
F|_{M_i}
\end{equation}
of $\Phi$.
\end{dfn}

We can easily compute this new invariant $\tr(F|_{M_i},\Phi|_{M_i})\in
\Comp$ as follows. Let $M_i= \bigsqcup_{\alpha \in A}M_{i,\alpha}$ be a
stratification of $M_i$ by connected subanalytic manifolds $M_{i,\alpha}$
such that $H^j(F)|_{M_{i,\alpha}}$ is a locally constant sheaf for any
$\alpha \in A$ and $j \in \Z$. Namely, we assume that the stratification
$M_i = \bigsqcup_{\alpha \in A} M_{i,\alpha}$ is adapted to $F|_{M_i}$.

\begin{dfn}\label{dfn:2-6}
For each $\alpha \in A$, we set
\begin{equation}
c_{\alpha}:= \dsum_{j\in \Z}(-1)^j \tr\{ H^j(F)_{x_{\alpha}}
\overset{\Phi|_{\{x_{\alpha}\}}}{\longrightarrow} H^j(F)_{x_{\alpha}}\}
\hspace{5mm}\in \Comp,
\end{equation}
where $x_{\alpha}$ is a reference point of $M_{i,\alpha}$.
\end{dfn}

Then we have the following very useful result due to Goresky-MacPherson.

\begin{prp}[\cite{G-M-1}]\label{prp:2-7}
In the situation as above, we have
\begin{equation}
\tr(F|_{M_i},\Phi|_{M_i})=\dsum_{\alpha \in A}c_{\alpha} \cdot
\chi_c(M_{i,\alpha}),
\end{equation}
where $\chi_c$ is the Euler-Poincar{\'e} index with compact supports.
\end{prp}

In terms of the theory of topological integrals of constructible functions
developed by Kashiwara-Schapira \cite{K-S}, Schapira \cite{Schapira} 
and Viro \cite{Viro} etc., we can
restate this result in the following way. Since we need $\Comp$-valued
constructible functions, we slightly generalize the usual notion of
$\Z$-valued constructible functions.

\begin{dfn}\label{dfn:2-8}
Let $Z$ be a subanalytic set. Then we say that a $\Comp$-valued function
$\varphi \colon Z \longrightarrow \Comp$ is constructible if there exists a
stratification $Z= \bigsqcup_{\alpha \in A} Z_{\alpha}$ of $Z$ by
subanalytic manifolds $Z_{\alpha}$ such that $\varphi|_{Z_{\alpha}}$ is a
constant function for any $\alpha \in A$. We denote by $\CF(Z)_{\Comp}$ the
abelian group of $\Comp$-valued constructible functions on $Z$.
\end{dfn}

Let $\varphi =\sum_{\alpha \in A} c_{\alpha} \cdot \1_{Z_{\alpha}} \in
\CF(Z)_{\Comp}$ be a $\Comp$-valued constructible function with compact
support on a subanalytic set $Z$, where $Z= \bigsqcup_{\alpha \in A}
Z_{\alpha}$ is a stratification of $Z$ and $c_{\alpha} \in \Comp$. Then we
can easily prove that the complex number $\sum_{\alpha \in A} c_{\alpha}
\cdot \chi_c(Z_{\alpha})$ does not depend on the expression
$\varphi=\sum_{\alpha \in A} c_{\alpha} \cdot \1_{Z_{\alpha}}$ of $\varphi$.

\begin{dfn}\label{dfn:2-9}
For a $\Comp$-valued constructible function $\varphi=\sum_{\alpha \in
A}c_{ \alpha} \cdot \1_{Z_{\alpha}} \in \CF(Z)_{\Comp}$ with compact support
as above, we set
\begin{equation}
\dint_Z\varphi :=\dsum_{\alpha \in A}c_{\alpha} \cdot \chi_c(Z_{\alpha}) \in
\Comp
\end{equation}
and call it the topological integral of $\varphi$.
\end{dfn}

By this definition, the result of Proposition \ref{prp:2-7} can be rewritten
as
\begin{equation}
\tr(F|_{M_i},\Phi|_{M_i}) =\dint_{M_i} \varphi(F|_{M_i},\Phi|_{M_i}),
\end{equation}
where the $\Comp$-valued constructible function
$\varphi(F|_{M_i},\Phi|_{M_i}) \in
\CF(M_i)_{\Comp}$ on $M_i$ is defined by
\begin{equation}\label{eq:2-23}
\varphi(F|_{M_i},\Phi|_{M_i})
(x):=\dsum_{j \in \Z} (-1)^j\tr\{ H^j(F)_x
\overset{\Phi|_{\{x\}}}{\longrightarrow} H^j(F)_x\}
\end{equation}
for $x\in M_i$.

Let us explain how the $\Comp$-valued constructible functions discussed
above are related to the theory of Lagrangian cycles in \cite[Chapter
IX]{K-S}. Now let $Z$ be a real analytic manifold and denote by $T^*Z$ its
cotangent bundle. Recall that Kashiwara-Schapira constructed the sheaf
$\L_Z$ of closed conic subanalytic Lagrangian cycles on $T^*Z$ in \cite{K-S}
(in this paper, we consider Lagrangian cycles with coefficients in $\Comp$).

\begin{prp}[\cite{K-S}]\label{prp:2-10}
There exists a group isomorphism
\begin{equation}
CC \colon \CF(Z)_{\Comp} \simto \varGamma(T^*Z;\L_Z)
\end{equation}
by which the characteristic function $\1_K$ of a closed submanifold $K
\subset Z$ of $Z$ is sent to the conormal cycle $[T^*_KZ]$ in $T^*Z$.
\end{prp}
We call $CC$ the characteristic cycle map in this paper. From now on, we fix
a fixed point component $M_i$ and always assume that $\supp(F) \cap M_i$ is
compact.

\begin{dfn}\label{dfn:3-1}
We say that the global trace $\tr(F,\Phi)$ is localizable to $M_i$ if the
equality
\begin{equation}
c(F,\Phi)_{M_i}=\tr(F|_{M_i},\Phi|_{M_i})
\end{equation}
holds.
\end{dfn}

By Proposition \ref{prp:2-7}, once the global trace is localizable to $M_i$,
the local contribution $c(F,\Phi)_{M_i}$ of $(F, \Phi)$ from $M_i$ can be
very easily computed. Let us denote $M_i$, $c(F,\Phi)_{M_i}$ etc.\ simply by
$M$, $c(F,\Phi)_M$ etc.\ respectively. From now on, we shall introduce some
useful criterions for the localizability of the global trace to $M$. First
let us consider the natural morphism
\begin{equation}
\phi^{\prime} \colon T_{M_{\reg}}X \longrightarrow T_{M_{\reg}}X
\end{equation}
induced by $\phi \colon X \longrightarrow X$, where $M_{\reg}$ denotes the
set of regular points in $M$. Since $M_{\reg}$ is not always connected in
the real analytic case, the rank of $T_{M_{\reg}}X$ may vary depending on
the connected components of $M_{\reg}$.

\begin{dfn}\label{dfn:3-2}
Let $V$ be a finite-dimensional vector space over $\Real$.
For its $\Real$-linear endomorphism
$A \colon V \longrightarrow V$, we set
\begin{align}
\Ev(A):=\{\text{the eigenvalues of } A^\Comp \colon
V^\Comp \longrightarrow V^\Comp \} \subset \Comp,
\end{align}
where $V^\Comp$ is the complexification of $V$.
\end{dfn}
In particular, for $x \in M_{\reg}$ we set
\begin{equation}
\Ev(\phi^{\prime}_x) :=\{ \text{the eigenvalues of $\phi^{\prime}_x \colon
(T_{M_{\reg}}X)_x \longrightarrow (T_{M_{\reg}}X)_x$} \} \subset \Comp.
\end{equation}

We also need the specialization functor
\begin{equation}
\nu_{M_{\reg}} \colon \Db(X) \longrightarrow \Db(T_{M_{\reg}}X)
\end{equation}
along $M_{\reg} \subset X$. In order to recall the construction of this
functor, consider the standard commutative diagram:
\begin{equation}
\xymatrix{
T_{M_{\reg}}X \ar[d]^{\tau} \ar@{^{(}->}[r]^s & \tl{X_{M_{\reg}}} \ar[d]^p &
\Omega_X \ar@{_{(}->}[l]_j \ar[dl]^{\tl{p}} \\
M_{\reg} \ar@{^{(}->}[r]^i& X,} \label{diag:2-1}
\end{equation}
where $\tl{X_{M_{\reg}}}$ is the normal deformation of $X$ along $M_{\reg}$
and $t \colon \tl{X_{M_{\reg}}} \longrightarrow \Real$ is the deformation
parameter. Recall that $\Omega_X$ is defined by $t > 0$ in
$\tl{X_{M_{\reg}}}$. Then the specialization $\nu_{M_{\reg}}(F)$ of $F$
along $M_{\reg}$ is defined by
\begin{equation}
\nu_{M_{\reg}}(F):=s^{-1}Rj_*\tl{p}^{-1}(F).
\end{equation}
Note that $\nu_{M_{\reg}}(F)$ is a conic object in $\Db(T_{M_{\reg}}X)$
whose support is contained in the normal cone $C_{M_{\reg}}(\supp(F))$ to
$\supp(F)$ along $M_{\reg}$. Since $F$ is $\Real$-constructible,
$\nu_{M_{\reg}}(F)$ is also $\Real$-constructible. By construction, there
exists a natural morphism
\begin{equation}
\Phi^{\prime} \colon (\phi^{\prime})^{-1}\nu_{M_{\reg}}(F) \longrightarrow
\nu_{M_{\reg}}(F)
\end{equation}
induced by $\Phi \colon \phi^{-1}F \longrightarrow F$. In the sequel, let us
assume the conditions:
\begin{enumerate}
\item $\supp(F) \cap M$ is compact and contained in $M_{\reg}$.
\item $1 \notin \Ev(\phi^{\prime}_x)$ for any $x \in \supp(F) \cap M_{\reg}$.
\end{enumerate}
The condition (ii) implies that the graph of $\phi$ in $X\times X$
intersects cleanly (see \cite[Definition 4.1.5]{K-S}) with the diagonal set
$\Delta_X \simeq X$ in an open neighborhood of $\supp(F) \cap M_{\reg}$. It
follows also from the condition (ii) that for an open neighborhood $U$ of
$\supp(F) \cap M_{\reg}$ in $M_{\reg}$ the fixed point set of
$\phi^{\prime}|_{\tau^{-1}(U)} \colon \tau^{-1}(U) \longrightarrow
\tau^{-1}(U)$ is contained in the zero-section $M_{\reg}$ of
$T_{M_{\reg}}X$. Set $\tl{U}=\tau^{-1}(U)$,
$\tl{F}=\nu_{M_{\reg}}(F)|_{\tl{U}}$ and $\tl{\Phi} =\Phi^{\prime}|_{\tl{U}}
\colon (\phi^{\prime}|_{\tl{U}})^{-1}\tl{F} \longrightarrow \tl{F}$. Then
also for the pair $(\tl{F},\tl{\Phi})$, we can define the characteristic
class $C(\tl{F},\tl{\Phi})\in H_{\supp(F) \cap
M_{\reg}}^0(\tl{U};\omega_{\tl{U}})$.

\begin{prp}\label{prp:3-3}{\rm \bf (\cite[Proposition 3.1]{M-T-3})}
Under the conditions {\rm(i)} and {\rm (ii)},
the local contribution $c(F,\Phi)_M$ from $M$ is
equal to $\dint_{\tl{U}}C(\tl{F},\tl{\Phi})$.
\end{prp}

In \cite[Theorem 3.2]{M-T-3} we proved the following result by Proposition
\ref{prp:3-3}.

\begin{thm}\label{thm:3-4}{\rm \bf(\cite[Theorem 3.2]{M-T-3})}
Under the conditions {\rm (i)} and {\rm (ii)}, assume moreover that
\begin{equation}\label{eq:3-6}
\Ev(\phi^{\prime}_x) \cap \Real_{> 1} =\emptyset
\end{equation}
for any $x \in \supp(F) \cap M \subset M_{\reg}$. Then the localization
\begin{equation}
c(F,\Phi)_M =\tr(F|_M,\Phi|_M) = 
\dint_M \varphi(F|_M, \Phi|_M)
\end{equation}
holds.
\end{thm}

In the complex case, we have the following stronger result.

\begin{thm}\label{thm:3-5} {\rm \bf (\cite[Theorem 3.3]{M-T-3}) }
Under the conditions {\rm (i)} and {\rm (ii)},
assume moreover that $X$ and $\phi \colon X
\longrightarrow X$ are complex analytic and $F \in \Db_c(X)$ i.e. $F$ is
$\Comp$-constructible. Assume also that there exists a compact complex
manifold $N$ such that $\supp(F) \cap M \subset N \subset M$. Then the
localization
\begin{equation}
c(F,\Phi)_M=\tr(F|_M,\Phi|_M)=
\dint_M \varphi(F|_M, \Phi|_M)
\end{equation}
holds.
\end{thm}

\begin{rem}\label{rem:3-6}
Later we will generalize 
Theorems \ref{thm:3-4} and \ref{thm:3-5}. To treat the more general
case where the set $\Ev(\phi^{\prime}_x)$ may vary depending 
on $x \in \supp(F) \cap M$, we need some precise arguments 
on Lefschetz cycles which will be introduced in
the next section. One naive idea to treat this case would be to cover
$\supp(F) \cap M$ by sufficiently small closed subsets $Z_i \subset \supp(F)
\cap M$ and use the local contributions of $(
\nu_{M_{\reg}}(F) )_{\tau^{-1}Z_i}$ to compute that of $\nu_{M_{\reg}}(F)$
by a Mayer-Vietoris type argument. However this very simple idea does not
work, because we cannot apply \cite[Proposition 9.6.2]{K-S} to constructible
sheaves with ``non-compact" support such as $(
\nu_{M_{\reg}}(F) )_{\tau^{-1}Z_i}$ to justify the Mayer-Vietoris type
argument.
\end{rem}

\section{Some properties of Lefschetz cycles}\label{sec:4}

In this section, we recall our construction of Lefschetz cycles in
\cite{M-T-3} and their standard properties. We inherit the notations in
Section \ref{sec:2}. Now assume that the fixed point set $M=\{ x \in X \ |\
\phi(x)=x\}$ of $\phi \colon X \longrightarrow X$ is a submanifold of $X$.
However here we do not assume that $M$ is connected. We also assume that
$\Delta_X$ intersects with $\Gamma_{\phi}=\{ (\phi(x),x) \in X \times X \ |\
x \in X\}$ cleanly along $M$ in $X \times X$. Identifying $\Gamma_{\phi}$
with $X$ by the second projection $X \times X \longrightarrow X$, we obtain
a natural identification $M=\Gamma_{\phi} \cap \Delta_X$. We also identify
$T^*_{\Delta_X}(X\times X)$ with $T^*X$ 
by the first projection $T^*(X\times X) \simeq T^*X \times T^*X
\longrightarrow T^*X$ as usual.

\begin{lem}{\rm \bf(\cite[Lemma 4.1]{M-T-3})}
The subset $T^*_{\Gamma_{\phi}}(X\times X) \cap T^*_{\Delta_X}(X\times X)$
of $(\Gamma_{\phi} \cap \Delta_X) \utimes{\Delta_X} T^*_{\Delta_X}(X\times
X) \simeq M \utimes{X}T^*X$ is a subbundle of
$M\utimes{X}T^*X$ (whose rank as a vector bundle 
may vary depending on the connected components of $M$).
\end{lem}

\begin{dfn}{\rm \bf (\cite[Definition 4.2]{M-T-3})}
We denote the subbundle $T^*_{\Gamma_{\phi}}(X\times X) \cap
T^*_{\Delta_X}(X\times X)$ of $M\utimes{X}T^*X$ by $\Lb$ and call it the
Lefschetz bundle associated with $\phi \colon X \longrightarrow X$.
\end{dfn}

\begin{prp}\label{prp:4-4}{\rm \bf(\cite[Proposition 4.2]{M-T-3})}
The natural surjective morphism \\$\rho \colon M\utimes{X}T^*X
\longtwoheadrightarrow T^*M$ induces an isomorphism $\Lb \simto T^*M$.
\end{prp}

\noindent From now on, by Proposition \ref{prp:4-4} we shall identify the
Lefschetz bundle $\Lb$ with $T^*M$.

Now let $F$ be an object of $\Dc(X)$ and $\Phi \colon \phi^{-1}F
\longrightarrow F$ a morphism in $\Dc(X)$. To these data $(F, \Phi)$, we can
associate a conic Lagrangian cycle in the Lefschetz bundle $\Lb \simeq T^*M$
as follows. Denote by $\pi_X \colon T^*X \longrightarrow X$ the natural
projection and recall that we have the functor
\begin{equation}
\mu_{\Delta_X} \colon \Db(X\times X) \longrightarrow
\Db(T^*_{\Delta_X}(X\times X))
\end{equation}
of microlocalization which satisfies
\begin{equation}
R\pi_{X*}\mu_{\Delta_X} \simeq \delta_X^! \simeq
\delta_X^{-1}\RG_{\Delta_X}.
\end{equation}
Recall also that the micro-support $\SS(F)$ of $F$ is a closed conic
subanalytic Lagrangian subset of $T^*X$ and the support of
$\mu_{\Delta_X}(F\boxtimes \D F)$ is contained in $\SS(F) \subset T^*X
\simeq T^*_{\Delta_X}(X\times X)$. Then we have a chain of natural
morphisms:
\begin{eqnarray}
R{\rm Hom}_{\Comp_X}(F,F)
&\simeq & \RG(X;\delta_X^!(F \boxtimes \D F)) \label{eq:4-6}\\
&\simeq & \RG_{\SS(F)}(T^*X;\mu_{\Delta_X}(F \boxtimes \D F)) \\
&\longrightarrow & \RG_{\SS(F)}(T^*X; \mu_{\Delta_X}(h_*h^{-1}(F \boxtimes
\D F))) \\
&\simeq & \RG_{\SS(F)}(T^*X;\mu_{\Delta_X}(h_*(\phi^{-1}F \otimes \D F))) \\
&\overset{\Phi}{\longrightarrow} & \RG_{\SS(F)}(T^*X; \mu_{\Delta_X}(h_*(F
\otimes \D F))) \\
& \longrightarrow &
\RG_{\SS(F)}(T^*X;\mu_{\Delta_X}(h_*\omega_X)).\label{eq:4-11}
\end{eqnarray}

\begin{lem}\label{lem:4-5}{\rm \bf(\cite[Lemma 4.4]{M-T-3})}
\begin{enumerate}
\item The support of $\mu_{\Delta_X}(h_*\omega_X)$ is contained in $\Lb$.
\item The restriction of $\mu_{\Delta_X}(h_*\omega_X)$ to $\Lb\simeq T^*M$
is isomorphic to $\pi_M^{-1}\omega_M$, where \\ $\pi_M \colon T^*M
\longrightarrow M$ is the natural projection.
\end{enumerate}
\end{lem}

By Lemma \ref{lem:4-5} there exists an isomorphism
\begin{equation}\label{eq:4-21}
\mu_{\Delta_X}(h_*\omega_X) \simeq (i_{\Lb})_*\pi_M^{-1}\omega_M,
\end{equation}
where $i_{\Lb} \colon \Lb \longhookrightarrow T^*X$ is the inclusion map. In
what follows, we sometimes omit the symbol $(i_{\Lb})_*$ in the above
identification \eqref{eq:4-21}. Combining the chain of morphisms
\eqref{eq:4-6}-\eqref{eq:4-11} with the isomorphism \eqref{eq:4-21}, we
obtain a morphism
\begin{equation}\label{eq:4-22}
{\rm Hom}_{\Db(X)}(F,F) \longrightarrow H^0_{\SS(F) \cap \Lb}( \Lb;
\pi_M^{-1}\omega_M).
\end{equation}

\begin{dfn}{\rm \bf(\cite[Definition 4.5]{M-T-3})}
We denote by $LC(F,\Phi)$ the image of \\$\id_F \in {\rm Hom}_{\Db(X)}(F,F)$
in $H^0_{\SS(F) \cap \Lb}(\Lb; \pi_M^{-1}\omega_M)$ by the morphism
\eqref{eq:4-22}.
\end{dfn}

\begin{lem}{\rm \bf(\cite[Lemma 4.6]{M-T-3})}
$\SS(F) \cap \Lb$ is contained in a closed conic subanalytic Lagrangian
subset of $\Lb \simeq T^*M$. 
\end{lem}

\begin{dfn}{\rm \bf (\cite[Definition 4.7]{M-T-3})}
Choose a closed conic subanalytic Lagrangian subset $\Lambda$ of $\Lb \simeq
T^*M$ such that $\SS(F) \cap \Lb \subset \Lambda$. We consider $LC(F,\Phi)$
as an element of $H^0_{\Lambda}(\Lb;\pi_M^{-1}\omega_M)$ and call it the
Lefschetz cycle associated with the pair $(F,\Phi)$.
\end{dfn}

As a basic property of Lefschetz cycles, we have the following homotopy
invariance. Let $I=[0,1]$ and let $\phi \colon X\times I \longrightarrow X$
be the restriction of a morphism of real analytic manifolds $X \times \Real
\longrightarrow X$. For $t \in I$, let $i_t \colon X \longhookrightarrow X
\times I$ be the injection defined by $x \longmapsto (x,t)$ and set $\phi_t
:=\phi \circ i_t \colon X \longrightarrow X$. Assume that the fixed point
set of $\phi_t$ in $X$ is smooth and does not depend on $t \in I$. We denote
this fixed point set by $M$. Let $F\in \Dc(X)$ and consider a morphism $\Phi
\colon \phi^{-1}F \longrightarrow p^{-1}F$ in $\Dc(X \times I)$, where $p
\colon X \times I \longrightarrow X$ is the projection. We set
\begin{equation}
\Phi_t :=\Phi|_{X\times \{t\}} \colon \phi_t^{-1}F \longrightarrow F
\end{equation}
for $t \in I$. We denote the Lefschetz bundle associated with $\phi_t$ by
$\Lb_t \simeq T^*M$.

\begin{prp}\label{prp:4-9}
Assume that $\supp(F) \cap M$ is compact and 
$\SS(F) \cap \Lb_t \subset T^*M$ does not depend on $t
\in I$ as a subset of $T^*M$. Then the Lefschetz cycle \\
$LC(F,\Phi_t) \in$ $H^0_{\SS(F) \cap
\Lb_t}(T^*M; \pi_M^{-1}\omega_M)$ does not depend on $t \in I$.
\end{prp}

\begin{proof}
The proof proceeds completely in the same way as that of \cite[Proposition
9.6.8]{K-S}. Hence we omit the detail.
\qed
\end{proof}

\section{Microlocal index formula for local contributions}\label{sec:5}

In this section, using the Lefschetz cycle $LC(F,\Phi)$ introduced in
Section \ref{sec:4}, we introduce our microlocal index theorem which
expresses local contributions of $(F,\Phi)$ as intersection numbers of the
images of continuous sections of $\Lb \simeq T^*M$ and $LC(F,\Phi)$. Here we
do not assume that the fixed point set $M$ of $\phi \colon X \longrightarrow
X$ is smooth. However we assume the condition:
\begin{equation}
1 \notin \Ev(\phi^{\prime}_x) \hspace{5mm} \text{for any $x \in M_{\reg}$}.
\end{equation}
Also in this more general setting, we can define the Lefschetz bundle $\Lb
\simeq T^*M_{\reg}$ over $M_{\reg}$ and construct the Lefschetz cycle
$LC(F,\Phi)$ in $\Lb$ by using the methods in Section \ref{sec:4}. Let $M =
\bigsqcup_{i \in I}M_i$ be the decomposition of $M$ into connected
components. Denote $(M_i)_{\reg}$ simply by $N_i$ and set $\Lb_i :=N_i
\utimes{M_{\reg}}\Lb$. Then we get a decomposition $\Lb = \bigsqcup_{i \in
I} \Lb_i \simeq \bigsqcup_{i \in I}T^*N_i$ of $\Lb$. By the direct sum
decomposition
\begin{equation}
H^0_{\SS(F) \cap \Lb}(\Lb; \pi_{M_{\reg}}^{-1}\omega_{M_{\reg}}) \simeq
\bigoplus_{i \in I}H^0_{\SS(F) \cap
\Lb_i}(\Lb_i;\pi_{N_i}^{-1}\omega_{N_i}),
\end{equation}
we obtain a decomposition
\begin{equation}
LC(F,\Phi) =\dsum_{i \in I} LC(F,\Phi)_{M_i}
\end{equation}
of $LC(F,\Phi)$, where $LC(F,\Phi)_{M_i} \in H^0_{\SS(F) \cap
\Lb_i}(\Lb_i;\pi_{N_i}^{-1}\omega_{N_i}) $. Now let us fix a fixed point
component $M_i$ and assume that $\supp(F) \cap M_i$ is compact and contained
in $N_i=(M_i)_{\reg}$. We shall show how the local contribution
$c(F,\Phi)_{M_i}\in \Comp$ of $(F,\Phi)$ from $M_i$ can be expressed by
$LC(F,\Phi)_{M_i}$. In order to state our results, for the sake of
simplicity, we denote $N_i=(M_i)_{\reg}$, $\Lb_i$, $LC(F,\Phi)_{M_i}$,
$c(F,\Phi)_{M_i}$ simply by $M$, $\Lb$, $LC(F,\Phi)$, $c(F,\Phi)$
respectively. Recall that to any continuous section $\sigma \colon M
\longrightarrow \Lb \simeq T^*M$ of the vector bundle $\Lb$, we can
associate a cycle $[\sigma] \in H^0_{\sigma(M)}(T^*M;\pi_M^!(\Comp_M))$
which is the image of $1 \in H^0(M;\Comp_M)$ by the isomorphism
$H^0_{\sigma(M)}(T^*M;\pi_M^!\Comp_M) \simeq H^0(M;(\pi_M\circ
\sigma)^!\Comp_M) \simeq H^0(M;\Comp_M)$ (see \cite[Definition 9.3.5]{K-S}).
If $\sigma(M) \cap \supp (LC(F,\Phi))$ is compact, we can define the
intersection number $\sharp ([\sigma] \cap LC(F,\Phi))$ of $[\sigma]$ and
$LC(F,\Phi)$ to be the image of $[\sigma] \otimes LC(F,\Phi)$ by the chain
of morphisms
\begin{eqnarray}
H^0_{\sigma(M)}(\Lb;\pi_M^!\Comp_M) \otimes H^0_{ \supp
(LC(F,\Phi))}(\Lb;\pi_M^{-1}\omega_M)
&\longrightarrow & H^0_{\sigma(M) \cap \supp (LC(F,\Phi))}
(\Lb;\omega_{\Lb}) \\
& \overset{\int_{\Lb}}{\longrightarrow} & \Comp.
\end{eqnarray}

\begin{thm}\label{thm:5-1}{\rm \bf(\cite[Theorem 4.8]{M-T-3})}
Assume that $\supp (F) \cap M$ is compact. Then for any continuous section
$\sigma \colon M \longrightarrow \Lb \simeq T^*M$ of $\Lb$, we have
\begin{equation}
c(F,\Phi) =\sharp ([\sigma] \cap LC(F,\Phi)).
\end{equation}
\end{thm}

As an application of Theorem \ref{thm:5-1}, we shall give a useful formula
which enables us to describe the Lefschetz cycle $LC(F,\Phi)$ explicitly in
the special case where $\phi \colon X \longrightarrow X$ is the identity map
of $X$ and $M=X$. For this purpose, until the end of this section, we shall
consider the situation where $\phi ={\rm id}_X$, $M=X$ and $\Phi \colon F
\longrightarrow F$ is an endomorphism of $F \in \Dc(X)$. In this case,
$LC(F,\Phi)$ is a Lagrangian cycle in $T^*X$. Now for real analytic function
$f \colon Y \longrightarrow I$ on a real analytic manifold $Y$ ($I$ is
an open interval in $\Real$) we define a section $\sigma_{f} \colon Y
\longrightarrow T^*Y$ of $T^*Y$ by $\sigma_{f}(y):=(y ;df(y))\
(y \in Y)$ and set
\begin{equation}
\Lambda_{f} :=\sigma_{f}(Y)=\{ (y ;df(y)) \ |\ y \in Y \}.
\end{equation}
Note that $\Lambda_{f}$ is a Lagrangian submanifold of $T^*Y$ Then we
have the following analogue of \cite[Theorem 9.5.3]{K-S}.

\begin{thm}\label{thm:5-2}
Let $Y$ be a real analytic manifold, $G$ an object of $\Dc(Y)$ and $\Psi
\colon G \longrightarrow G$ an endomorphism of $G$. For a real analytic
function $f \colon Y \longtwoheadrightarrow I$, assume that the
following conditions are satisfied.
\begin{enumerate}
\item $\supp (G) \cap \{ y \in Y\ |\ f(y) \leq t\}$ is compact for any
$t \in I$.
\item $\SS(G) \cap \Lambda_{f}$ is compact.
\end{enumerate}
Then the global trace
\begin{equation}
\tr(G,\Psi)=\dsum_{j\in \Z}(-1)^j \tr \{ H^j(Y;G)
\overset{\Psi}{\longrightarrow}H^j(Y;G)\}
\end{equation}
of $(G,\Psi)$ is equal to $\sharp ([\sigma_{f}] \cap LC(G,\Psi))$.
\end{thm}

\begin{proof}
Since the fixed point set of $\phi =\id_Y$ is $Y$ itself, $LC(G,\Psi)$ is a
Lagrangian cycle in $T^*Y$. Moreover, since any open subset of $Y$ is
invariant by $\phi=\id_Y$, we can freely use the microlocal Morse lemma
(\cite[Corollary 5.4.19]{K-S}) to reduce the computation of the global trace
$\tr(G,\Psi)$ on $Y$ to that of
\begin{equation}
\dsum_{j\in \Z}(-1)^j \tr \{ H^j(\Omega_t;G)
\overset{\Psi|_{\Omega_t}}{\longrightarrow}H^j(\Omega_t;G)\}
\end{equation}
for sufficiently large $t>0$ in $I$, where we set $\Omega_t:=\{y \in Y \ |\
f(y)<t\}$. Then the proof proceeds essentially in the same way as that
of \cite[Theorem 9.5.3]{K-S}.
\qed
\end{proof}

\begin{thm}\label{thm:5-3}
Let $X$, $F \in \Dc(X)$ and $\Phi \colon F \longrightarrow F$ be as above.
For a real analytic function $f \colon X \longrightarrow \Real$ and a
point $x_0 \in X$, assume the condition
\begin{equation}
\Lambda_{f} \cap \SS(F) \subset \{(x_0;df(x_0))\}.
\end{equation}
Then the intersection number $\sharp ([\sigma_{f}] \cap LC(F,\Phi))$
(at the point $(x_0;df(x_0))\in T^*X$) is equal to
\begin{equation}
\dsum_{j\in \Z}(-1)^j \tr \{ H^j_{\{f \geq f(x_0)\}}(F)_{x_0}
\overset{\Phi}{\longrightarrow}H^j_{\{f \geq f(x_0)\}}(F)_{x_0}
\}.
\end{equation}
\end{thm}

\begin{proof}
The proof is very similar to that of \cite[Theorem 9.5.6]{K-S}. For a
sufficiently small open ball $B(x_0,\e)=\{ x \in X \ |\ |x-x_0|<\e\}$
centered at $x_0$, set $F_0=\RG_{B(x_0,\e)}(F)\in \Dc(X)$. Then $\Phi$
induces a natural morphism $\Phi_0 \colon F_0 \longrightarrow F_0$ in
$\Dc(X)$. Moreover by the proof of \cite[Theorem 9.5.6]{K-S}, we have
\begin{equation}\label{eq:5-13}
\Lambda_{f} \cap \SS(F_0 ) \subset \pi_X^{-1}(\Omega_{-t}) \sqcup
\{(x_0;df(x_0))\}
\end{equation}
for sufficiently small $t >0$, where we set $\Omega_k:=\{x \in X\ |\
f(x) -f(x_0) <k\}$ for $k \in \Real$. Then applying Theorem
\ref{thm:5-2} to the case where $I=(-\infty, 0)$, $Y=\Omega_0$,
$G=F_0|_{\Omega_0} \in \Dc(Y)$ and $\Psi=\Phi_0|_{\Omega_0} \colon G
\longrightarrow G$, we obtain
\begin{eqnarray}
\lefteqn{\sharp ([\sigma_{f}] \cap LC(F_0 ,\Phi_0 ) \cap
\pi_X^{-1}(\Omega_0)) } \nonumber \\
&=& \dsum_{j \in \Z} (-1)^j \tr \{H^j(B(x_0,\e) \cap \Omega_0;F)
\overset{\Phi}{\longrightarrow} H^j(B(x_0,\e) \cap \Omega_0;F)
\}.\label{eq:5-14}
\end{eqnarray}
On the other hand, since $\supp (F_0 )$ is compact in $X$, by Theorem
\ref{thm:5-1} we have
\begin{equation}\label{eq:5-15}
\sharp ([\sigma_{f}] \cap LC(F_0 ,\Phi_0 ))= \dsum_{j \in \Z}
(-1)^j\tr \{H^j(B(x_0,\e) ;F) \overset{\Phi}{\longrightarrow}
H^j(B(x_0,\e);F) \}.
\end{equation}
Comparing \eqref{eq:5-14} with \eqref{eq:5-15} in view of \eqref{eq:5-13},
we see that the intersection number of $[\sigma_{f}]$ and $LC(F_0
,\Phi_0 )$ at $(x_0;df(x_0))$ is equal to
\begin{equation}\label{eq:5-17}
\dsum_{j \in \Z} (-1)^j \tr\{ H^j_{\{ f \geq f(x_0)\}} (F)_{x_0}
\overset{\Phi}{\longrightarrow} H^j_{\{ f \geq f(x_0)
\}}(F)_{x_0} \}.
\end{equation}
Since $LC(F,\Phi)= LC(F_0 ,\Phi_0 )$ in an open neighborhood of
$(x_0;df(x_0))$ in $T^*X$, this last intersection number $\sharp
([\sigma_{f}]\cap LC(F_0, \Phi_0))$ ($=$\eqref{eq:5-17}) is equal to
$\sharp ([\sigma_{f}] \cap LC(F,\Phi))$. This completes the proof.
\qed
\end{proof}

By Theorem \ref{thm:5-3}, we can explicitly describe the Lefschetz cycle
\\$LC(F,\Phi) \in \varGamma(T^*X;\L_X)$ as follows. Let $X=
\bigsqcup_{\alpha \in A}X_{\alpha}$ be a $\mu$-stratification of $X$ such
that
\begin{equation}
\supp(LC(F,\Phi)) \subset \SS(F) \subset \bigsqcup_{\alpha \in A}
T_{X_{\alpha}}^*X.
\end{equation}
Then $\Lambda := \bigsqcup_{\alpha \in A} T_{X_{\alpha}}^*X$ is a closed
conic subanalytic Lagrangian subset of $T^*X$. Moreover there exists an open
dense smooth subanalytic subset $\Lambda_0$ of $\Lambda$ whose decomposition
$\Lambda_0=\bigsqcup_{i \in I}\Lambda_i$ into connected components satisfies
the condition
\begin{equation}\label{eq:5-18}
``\text{For any $i\in I$, there exists $\alpha_i \in A$ such that $\Lambda_i
\subset T_{X_{\alpha_i}}^*X$. }"
\end{equation}

\begin{dfn}\label{dfn:5-4}
For $i \in I$ and $\alpha_i\in A$ as above, we define a complex number $m_i
\in \Comp$ by
\begin{equation}\label{eq:5-19}
m_i :=\dsum_{j \in \Z} (-1)^j \tr\{H^j_{\{f \geq f(x)\}}(F)_x
\overset{\Phi}{\longrightarrow} H^j_{\{f \geq f(x)\}}(F)_x\},
\end{equation}
where the point $x \in \pi_X(\Lambda_i) \subset X_{\alpha_i}$ and the
$\Real$-valued real analytic function $f \colon X \longrightarrow
\Real$ (defined in an open neighborhood of $x$ in $X$) are defined as
follows. Take a point $p \in \Lambda_i$ and set $x=\pi_X(p) \in
X_{\alpha_i}$. Then $f \colon X \longrightarrow \Real$ is a real
analytic function which satisfies the following conditions:
\begin{enumerate}
\item $p=(x;df(x)) \in \Lambda_i$.
\item The Hessian ${\rm Hess} (f|_{X_{\alpha_i}})$ of
$f|_{X_{\alpha_i}}$ is positive definite.
\end{enumerate}
\end{dfn}

\begin{cor}\label{cor:5-5}
In the situation as above, for any $i \in I$ there exists an open
neighborhood $U_i$ of $\Lambda_i$ in $T^*X$ such that
\begin{equation}
LC(F,\Phi)=m_i \cdot [T_{X_{\alpha_i}}^*X]
\end{equation}
in $U_i$.
\end{cor}

Now let us define a $\Comp$-valued constructible function $\varphi(F,\Phi)$
on $X$ by
\begin{equation}
\varphi(F,\Phi)(x):=\dsum_{j \in \Z} (-1)^j\tr\{ H^j(F)_x
\overset{\Phi|_{\{x\}}}{\longrightarrow} H^j(F)_x\}
\end{equation}
for $x\in X$. We will show that the characteristic cycle
$CC(\varphi(F,\Phi))$ of $\varphi(F,\Phi)$ (see Proposition \ref{prp:2-10})
is equal to the Lefschetz cycle $LC(F,\Phi)$. For this purpose, we need the
following.

\begin{dfn}[\cite{K-S} and \cite{M-T-1}] 
Let $\varphi \colon X \longrightarrow \Z$ be a $\Z$-valued constructible
function on $X$ and $U$ a relatively compact subanalytic open subset in $X$.
We define the topological integral $\dint_U \varphi$ of $\varphi$ over $U$
by 
\begin{equation}
\dint_U \varphi =\dsum_{\alpha \in \Z} c_{\alpha} \cdot
\chi(\RG(U;\Comp_{X_{\alpha}})),
\end{equation}
where $\varphi =\sum_{\alpha \in A} c_{\alpha} \1_{X_{\alpha}}$ ($c_{\alpha}
\in \Z$) is an expression of $\varphi$ with respect to a subanalytic
stratification $X=\bigsqcup_{\alpha \in A} X_{\alpha}$ of $X$.
\end{dfn}

We can extend $\Comp$-linearly this integral $\dint_U \colon \CF(X)
\longrightarrow \Z$ and obtain a $\Comp$-linear map
\begin{equation}
\dint_U \colon \CF(X)_{\Comp} \longrightarrow \Comp.
\end{equation}
On the other hand, since any relatively compact subanalytic open subset $U$
of $X$ is invariant by $\phi=\id_X$, the global trace on $U$
\begin{equation}
\tr(F|_U,\Phi|_U)=\dsum_{j \in \Z}(-1)^j\tr\{H^j(U;F)
\overset{\Phi|_U}{\longrightarrow}H^j(U;F)\}
\end{equation}
is well-defined.

\begin{lem}\label{lem:5-7}
For any relatively compact subanalytic open subset $U$ of $X$, we have
\begin{equation}
\tr(F|_U,\Phi|_U) =\dint_U \varphi(F,\Phi).
\end{equation}
\end{lem}

The proof of this lemma being completely similar to that of
\cite[Proposition 11.6]{G-M-1}, we omit the proof.

\begin{thm}\label{thm:5-8}
In the situation $\phi=\id_X$, $\Phi \colon F \longrightarrow F$ etc. as
above, we have the equality
\begin{equation}
LC(F,\Phi)=CC(\varphi(F,\Phi))
\end{equation}
as Lagrangian cycles in $T^*X$.
\end{thm}

\begin{proof}
Let $X=\bigsqcup_{\alpha \in A}X_{\alpha}$ be a $\mu$-stratification of $X$
such that
\begin{equation}
\supp (LC(F,\Phi)), \hspace{3mm}\supp(CC(\varphi(F,\Phi))) \subset \Lambda =
\bigsqcup_{\alpha \in A}T_{X_{\alpha}}^*X.
\end{equation}
Take an open dense smooth subanalytic subset $\Lambda_0 $ of $\Lambda$ whose
decomposition $\Lambda= \bigsqcup_{i \in I}\Lambda_i$ into connected
components satisfies the condition \eqref{eq:5-18}. Let us fix $\Lambda_i$
and $X_{\alpha_i}$ such that $\Lambda_i \subset T_{X_{\alpha_i}}^*X$. It is
enough to show that $LC(F,\Phi)$ and $CC(\varphi(F,\Phi))$ coincide in an
open neighborhood of $\Lambda_i$ in $T^*X$. By Corollary \ref{cor:5-5}, in
an open neighborhood $U_i$ of $\Lambda_i$ in $T^*X$ we have
\begin{equation}
LC(F,\Phi)=m_i \cdot [T_{X_{\alpha_i}}^*X],
\end{equation}
where $m_i \in \Comp$ is defined by \eqref{eq:5-19} for $p \in \Lambda_i$,
$x=\pi_X(p)\in X_{\alpha_i}$, $f \colon X \longrightarrow \Real$ as in
Definition \ref{dfn:5-4}. Let $U$ be a sufficiently small open ball in $X$
centered at $x \in X_{\alpha_i}$. Set $V:=U \cap \{ f <f(x)\}$.
Then we have
\begin{eqnarray}
m_i
&=& \dsum_{j \in \Z}(-1)^j\tr\{H^j_{\{f \geq f(x)\}}(U;F)
\overset{\Phi}{\longrightarrow} H^j_{\{f \geq f(x)\}}(U;F) \} \\
&=& \tr(F|_U,\Phi|_U)-\tr(F|_V,\Phi|_V) \\
&=& \dint_U \varphi(F,\Phi) -\dint_V\varphi(F,\Phi).
\end{eqnarray}
This last number coincides with the coefficient of
$[T_{X_{\alpha_i}}^*X]|_{U_i}$ in $CC(\varphi(F,\Phi))|_{U_i}$. This
completes the proof. \qed
\end{proof}

\section{Hyperbolic localization and Lefschetz cycles}\label{sec:6}

In this section, we explicitly describe 
the Lefschetz cycle $LC(F,\Phi)$
introduced in Section \ref{sec:4} in terms of 
hyperbolic localizations of the specializations 
of $F$. Let $M$ be a possibly
singular fixed point component of $\phi \colon X \longrightarrow X$.
Throughout this section, we assume the conditions $\supp (F) \cap M 
\subset M_{\reg}$ and 
\begin{equation}
``\text{$1\notin \Ev(\phi^{\prime}_x)$ for 
any $x \in \supp(F) \cap M_{\reg}$.}"
\end{equation}
Then there exists an open neighborhood $U$ of $\supp(F) \cap M_{\reg}$ in
$M_{\reg}$ such that $\Gamma_{\phi}$ intersects with $\Delta_X$ cleanly
along $U \subset M \subset \Gamma_{\phi} \cap \Delta_X$. Namely, there
exists a Lefschetz bundle $\Lb=U \utimes{M}\{T^*_{\Gamma_{\phi}}(X\times X)
\cap T^*_{\Delta_X}(X\times X)\}$ over $U$ which is isomorphic to $T^*U$. As
in the same way as in Section \ref{sec:4}, we can define a Lagrangian cycle
in $\Lb$ associated with $(F,\Phi)$. We still denote it by $LC(F,\Phi)$ and
want to describe it explicitly. Replacing $X$, $M$ etc. by $X\setminus (M
\setminus U)$, $U$ etc. respectively, we may assume that $M$ is smooth and
$1 \notin \Ev(\phi^{\prime}_x)$ for any 
$x \in M$ from the first. In this situation, the
fixed point set of $\phi^{\prime} \colon T_MX \longrightarrow T_MX$ is the
zero-section $M$. Let $\Gamma_{\phi^{\prime}}=\{(\phi^{\prime}(p),p) \ |\ p
\in T_MX\}\subset T_MX \times T_MX$ be the graph of $\phi^{\prime}$ and
$\Delta_{T_MX}\simeq T_MX$ the diagonal subset of $T_MX \times T_MX$. Then
\begin{equation}
\Lb^{\prime}:=T^*_{\Gamma_{\phi^{\prime}}}(T_MX\times T_MX) \cap
T^*_{\Delta_{T_MX}}(T_MX \times T_MX)
\end{equation}
is a vector bundle over the zero-section $M \simeq \Gamma_{\phi^{\prime}}
\cap \Delta_{T_MX}$ of $T_MX$. Since $\Lb^{\prime}$ is also isomorphic to
$T^*M$ by our assumptions, we shall identify it with the original Lefschetz
bundle $\Lb=T^*_{\Gamma_{\phi}}(X \times X) \cap T^*_{\Delta_X}(X\times X)$.
Now consider the natural morphism
\begin{equation}
\Phi^{\prime} \colon (\phi^{\prime})^{-1}\nu_M(F) \longrightarrow \nu_M(F)
\end{equation}
induced by $\Phi \colon \phi^{-1}F \longrightarrow F$. Then from the pair
$(\nu_M(F), \Phi^{\prime})$, we can construct the Lefschetz cycle
$LC(\nu_M(F),\Phi^{\prime})$ in $\Lb' \simeq \Lb$.

\begin{prp}\label{prp:6-1}
In $\Lb \simeq \Lb^{\prime}$, we have
\begin{equation}
LC(F,\Phi) =LC(\nu_M(F),\Phi^{\prime}).
\end{equation}
\end{prp}

\begin{proof}
First, we briefly recall the proof of Proposition \ref{prp:3-3} which is
similar to that of \cite[Proposition 9.6.11]{K-S}. Since the construction of
the characteristic class $C(F,\Phi)_M \in H_{\supp(F) \cap M}^0(X;\omega_X)$
is local around $\supp(F) \cap M$ (see \cite[Remark 9.6.7]{K-S}) and $X
\setminus (M \setminus U)$ is invariant by $\phi$, we may replace $X$, $M$
etc. by $X\setminus (M \setminus U)$, $U$ etc. respectively. Then the proof
follows from the commutativity of the diagram \eqref{diag:3-6} below. Here
we denote $T_MX$ simply by $\Vb$ and the morphism $\tl{h} \colon T_MX
\longrightarrow T_MX \times T_MX$ is defined by
$\tl{h}=(\phi^{\prime},\id)$. We also used the natural isomorphism $\D
\nu_M(F) \simeq \nu_M(\D F)$. Let us explain the construction of the
morphism ${\bf A}$ in the diagram \eqref{diag:3-6}. Consider the commutative
diagram:

\begin{equation}
\xymatrix@R=3mm@C=5mm{
T_{M\times M}(X\times X) \ar@{^{(}->}[rr]^{s_1} & & \tl{(X\times X)_{M\times
M}} & & \Omega_{X\times X} \ar@{_{(}->}[ll]_{j_1} \ar@{->>}[rr]^{\tl{p_1}}&
& X\times X \\
& \Box & & \Box & & \Box & \\
T_MX \ar@{^{(}->}[rr]^{s} \ar@{^{(}->}[uu]^{\delta_{T_MX}} & & \tl{X_M}
\ar@{^{(}->}[uu]^{\tl{\delta^{\prime}}} & & \Omega_X \ar@{_{(}->}[ll]_{j}
\ar@{->>}[rr]^{\tl{p}} \ar@{^{(}->}[uu]^{\tl{\delta}} & & X,
\ar@{^{(}->}[uu]^{\delta_X}}
\end{equation}
where $\tl{(X\times X)_{M \times M}}$ is the normal deformation of $X\times
X$ along $M\times M$ and \\$t_1 \colon \tl{(X\times X)_{M\times M}}
\longrightarrow \Real$ is the deformation parameter such that
$\Omega_{X\times X}$ is defined by $t_1 > 0$ in $\tl{(X\times X)_{M \times
M}}$. Then the morphism ${\bf A}$ is constructed by the morphisms of
functors
\begin{eqnarray}
\delta_X^!
&\longrightarrow & \delta_X^!R\tl{p_1}_*\tl{p_1}^{-1}
\simeq R\tl{p}_* \tl{\delta}^!\tl{p_1}^{-1}
\simeq Rp_* \tl{\delta^{\prime}}^! Rj_{1*}\tl{p_1}^{-1} \nonumber\\
&\longrightarrow & Rp_* \tl{\delta^{\prime}}^!
s_{1*}s_1^{-1}Rj_{1*}\tl{p_1}^{-1}
\simeq  Rp_*s_* \delta_{T_MX}^!s_1^{-1}Rj_{1*}\tl{p_1}^{-1}.
\end{eqnarray}
The other horizontal arrows in the diagram \eqref{diag:3-6} are constructed
similarly.
{\footnotesize
\begin{equation}\label{diag:3-6}
\xymatrix@R=8mm@C=4mm{
R{\rm Hom}(F,F) \ar[rr] & & R{\rm Hom}(\nu_M(F),\nu_M(F)) \\
\RG_{\Delta_X}(X\times X;F \boxtimes \D F) \ar[u]^{\wr}
\ar[r]^{\hspace*{-5mm}{\bf A}} \ar[dd] & \RG_{\Delta_{\Vb}}(\Vb\times
\Vb;\nu_{M\times M}(F \boxtimes \D F)) \ar[d]& \RG_{\Delta_{\Vb}}(\Vb\times
\Vb;\nu_M(F) \boxtimes \D \nu_M(F)) \ar[l] \ar[d] \ar[u]^{\wr}\\
& \RG_{M}(\Vb;\tl{h}^{-1}\nu_{M\times M}(F \boxtimes \D F)) \ar[d] &
\RG_{M}(\Vb;\phi^{\prime -1}\nu_M(F) \otimes \D \nu_M(F)) \ar[l] \ar[d] \\
\RG_M(X; \phi^{-1}F \otimes \D F) \ar[d]^{\Phi} \ar[r]^{\sim} & \RG_M(\Vb;
\nu_M(\phi^{-1}F \otimes \D F)) \ar[d]^{\Phi} & \RG_M(\Vb;\nu_M(\phi^{-1}F)
\otimes \D \nu_M(F)) \ar[l] \ar[d]^{\Phi} \\
\RG_M(X;F\otimes \D F) \ar[d] \ar[r]^{\sim} & \RG_M(\Vb;\nu_M(F\otimes \D
F)) \ar[d] & \RG_M(\Vb;\nu_M(F) \otimes \D \nu_M(F)) \ar[l] \ar[d] \\
\RG_M(X;\omega_X) \ar[d] \ar[r]^{\sim} & \RG_M(\Vb;\nu_M(\omega_X)) &
\RG_M(\Vb;\omega_{\Vb}) \ar[d] \ar@{-}[l]_{\sim}\\
\Comp \ar@{=}[rr]& & \Comp. }
\end{equation}}

Now the proof of Proposition \ref{prp:6-1} follows from the commutativity of
Diagram \ref{prp:6-1}.a below, which is a microlocal version of Diagram
\eqref{diag:3-6}. Here we denote $T_MX$, $\SS(F)$ and $C_{T_M^*X}(\SS(F))$
by $\Vb$, $S$ and $S^{\prime}$ respectively. Note that we have natural
isomorphisms
\begin{equation}
T^*(T_MX)  \simeq  T^*(T^*_MX)  \simeq T_{T_M^*X}(T^*X)
\end{equation}
(see \cite[(6.2.3)]{K-S} and \eqref{eq:new} below) and the normal cone
$S^{\prime}=C_{T_M^*X}(\SS(F))$ can be considered as a subset of
$T^*(T_MX)=T^*\Vb$. We also used a conic isotropic subset $S^{\prime
\prime}=(S \cap \Lb) \cup (S^{\prime} \cap \Lb^{\prime})$ of $\Lb \simeq
\Lb^{\prime} \simeq T^*M$ and the morphism $\tl{h} \colon T_MX
\longrightarrow T_MX \times T_MX$ is defined by $\tl{h}=
(\phi^{\prime},\id_{T_MX})$. Moreover we used the natural isomorphism $\D
\nu_M(F) \simeq \nu_M(\D F)$ to obtain 
Diagram \ref{prp:6-1}.a. Let us explain the
construction of the morphism ${\bf A}$ in Diagram \ref{prp:6-1}.a. 
First consider the
commutative diagram:
\begin{equation}
\xymatrix@R=3mm@C=5mm{
T_{M\times M}(X\times X) \ar@{^{(}->}[rr]^{s_1} & & \tl{(X\times X)_{M\times
M}} & & \Omega_{X\times X} \ar@{_{(}->}[ll]_{j_1} \ar@{->>}[rr]^{\tl{p_1}}&
& X\times X \\
& \Box & & \Box & & \Box & \\
T_MX \ar@{^{(}->}[rr]^{s} \ar@{^{(}->}[uu]^{\delta_{T_MX}} & & \tl{X_M}
\ar@{^{(}->}[uu]^{\tl{\delta^{\prime}}} & & \Omega_X \ar@{_{(}->}[ll]_{j}
\ar@{->>}[rr]^{\tl{p}} \ar@{^{(}->}[uu]^{\tl{\delta}} & & X
\ar@{^{(}->}[uu]^{\delta_X}}
\end{equation}
which already appeared in the proof of Proposition \ref{prp:3-3}. Denote the
image of $\tl{\delta^{\prime}}$ (resp.\ $\tl{\delta}$) by $\Delta_{\tl{X_M}}$
(resp.\ $\Delta_{\Omega_X}$). Then we see that the following morphisms are
isomorphisms.
\begin{eqnarray}
\t{\tl{p_1}} &\colon &\Delta_{\Omega_X} \utimes{\Delta_X}
T_{\Delta_X}^*(X\times X) \longrightarrow
T_{\Delta_{\Omega_X}}^*\Omega_{X\times X}, \\
\t{j_1} &\colon &\Delta_{\Omega_X}
\utimes{\Delta_{\tl{X_M}}}T_{\Delta_{\tl{X_M}}}^*(\tl{(X\times X)_{M\times
M}}) \longrightarrow T_{\Delta_{\Omega_X}}^*\Omega_{X\times X}, \\
\t{s_1} &\colon &\Delta_{T_MX}
\utimes{\Delta_{\tl{X_M}}}T_{\Delta_{\tl{X_M}}}^*(\tl{(X\times X)_{M\times
M}}) \longrightarrow T_{\Delta_{T_MX}}^*(T_{M\times M}(X\times X)).
\end{eqnarray}

\newpage
\begin{center}{
\rotatebox[origin=c]{90}{
$
\xymatrix@R=10mm@C=10mm{
R{\rm Hom}(F,F) \ar[rr] & & R{\rm Hom}(\nu_M(F),\nu_M(F)) \\
\RG_{S}(T^*X;\mu_{\Delta_X}(F \boxtimes \D F)) \ar@{-}[u]^{\wr} \ar[r]^{{\bf
A}} \ar[dd] & \RG_{S^{\prime}}(T^*\Vb;\mu_{\Delta_{\Vb}}(\nu_{M\times M}(F
\boxtimes \D F))) \ar[d]&
\RG_{S^{\prime}}(T^*\Vb;\mu_{\Delta_{\Vb}}(\nu_M(F) \boxtimes \D \nu_M(F)))
\ar[l] \ar[d] \ar@{-}[u]^{\wr}\\
& \RG_{S^{\prime}}(T^*\Vb;\mu_{\Delta_{\Vb}}(\tl{h}_*\tl{h}^{-1}\nu_{M\times
M}(F \boxtimes \D F))) \ar[d]&
\RG_{S^{\prime}}(T^*\Vb;\mu_{\Delta_{\Vb}}(\tl{h}_*(\phi^{\prime -1}\nu_M(F)
\otimes \D \nu_M(F)))) \ar[l] \ar[d] \\
\RG_{S}(T^*X; \mu_{\Delta_X}(h_*(\phi^{-1}F \otimes \D F))  \ar[d]^{\Phi}
\ar[r]^{{\bf B}} & \RG_{S^{\prime}}(T^*\Vb;\mu_{\Delta_{\Vb}}(\tl{h}_*
\nu_M(\phi^{-1}F \otimes \D F))) \ar[d]^{\Phi} &
\RG_{S^{\prime}}(T^*\Vb;\mu_{\Delta_{\Vb}}(\tl{h}_* (\nu_M(\phi^{-1}F)
\otimes \D\nu_M(F)))) \ar[l] \ar[d]^{\Phi} \\
\RG_{S}(T^*X;\mu_{\Delta_X}(h_*(F\otimes \D F))) \ar[d] \ar[r] &
\RG_{S^{\prime}}(T^*\Vb;\mu_{\Delta_{\Vb}}(\tl{h}_*\nu_M(F\otimes \D F))
\ar[d] & \RG_{S^{\prime}}(T^*\Vb;\mu_{\Delta_{\Vb}}(\tl{h}_*(\nu_M(F)
\otimes \D \nu_M(F)))) \ar[l] \ar[d] \\
\RG_{S}(T^*X;\mu_{\Delta_X}(h_*\omega_X)) \ar[r] \ar[d] &
\RG_{S^{\prime}}(T^*\Vb;\mu_{\Delta_{\Vb}}(\tl{h}_*\nu_M(\omega_X))) \ar[d]
&  \RG_{S^{\prime}}(T^*\Vb;\mu_{\Delta_{\Vb}}(\tl{h}_*\omega_{\Vb}))
\ar@{-}[l]^{\sim} \ar[d] \\
\RG_{S^{\prime \prime}}(\Lb; \pi_M^{-1}\omega_M) \ar@{=}[r]& \RG_{S^{\prime
\prime}}(\Lb;\pi_M^{-1}\omega_M) \ar@{=}[r] & \RG_{S^{\prime
\prime}}(\Lb;\pi_M^{-1}\omega_M).}$}}

\bigskip
Diagram \ref{prp:6-1}.a
\end{center}

\newpage
Now let us set
\begin{eqnarray}
S_1&:=&\t{\tl{p_1}}(\Delta_{\Omega_X}\utimes{\Delta_X}S), \\
S_2&:=&\overline{\t{j_1}^{-1}(S_1)}, \\
S_3&:=&S_2 \cap T_{\Delta_{T_MX}}^*(T_{M\times M}(X\times X)).
\end{eqnarray}
Then we have the following morphisms
\begin{eqnarray}
\lefteqn{\RG_{\SS(F)} (T_{\Delta_X}^*(X\times X); \mu_{\Delta_X}(F \boxtimes
\D F))} \nonumber \\
&\longrightarrow & \RG_{\SS(F)}(T_{\Delta_X}^*(X\times X);
\mu_{\Delta_X}(R\tl{p_1}_*\tl{p_1}^{-1}(F \boxtimes \D F))) \\
&\longrightarrow & \RG_{S_1}(T_{\Delta_{\Omega_X}}^*\Omega_{X\times X};
\mu_{\Delta_{\Omega_X}}(\tl{p_1}^{-1}(F \boxtimes \D F))) \\
\label{eq:6-11}&\simot& \RG_{S_2}(T_{\Delta_{\tl{X_M}}}^*(\tl{(X\times
X)_{M\times M}});\mu_{\Delta_{\tl{X_M}}}(Rj_{1*}\tl{p_1}^{-1}(F \boxtimes \D
F)))\\
&\longrightarrow & \RG_{S_2}(T_{\Delta_{\tl{X_M}}}^*(\tl{(X\times
X)_{M\times M}});\mu_{\Delta_{\tl{X_M}}}
(s_{1*}s_1^{-1}Rj_{1*}\tl{p_1}^{-1}(F \boxtimes \D F))) \\
&\longrightarrow & \RG_{S_3}(T_{\Delta_{T_MX}}^*(T_{M\times M}(X\times
X));\mu_{\Delta_{T_MX}}(s_1^{-1}Rj_{1*}\tl{p_1}^{-1}(F \boxtimes \D F)))\\
& =& \RG_{S_3}(T^* \Vb ;\mu_{\Delta_{\Vb}}(\nu_{M \times M}((F \boxtimes \D
F)))),
\end{eqnarray}
where we used \cite[Theorem 4.3.2 and Proposition 3.3.9]{K-S} (see also the
arguments in \cite[page 192-193]{K-S}) to prove that the morphism
\eqref{eq:6-11} is an isomorphism. Let us show that $S_3$ is equal to
$S^{\prime}$. Let $(x^{\prime},x^{\prime \prime})$ be a local coordinate
system of $X$ such that $M=\{ x^{\prime}=0\}$ and $(x^{\prime},x^{\prime
\prime}; \xi^{\prime},\xi^{\prime \prime})$ the associated coordinates of
$T^*X$. Then by the Hamiltonian isomorphism etc., we can naturally identify
$T^*(T_MX) \simeq T_{\Delta_{T_MX}}^* (T_{M\times M}(X\times X))$ with
$T_{T_M^*X}(T^*X)$ as follows (see \cite[(6.2.3)]{K-S}).
\begin{equation}\label{eq:new}
\begin{array}{ccccc}
T^*(T_MX) & \simeq & T^*(T^*_MX) & \simeq & T_{T_M^*X}(T^*X). \\
\inun & & \inun & & \inun \\
(x^{\prime},x^{\prime \prime};\xi^{\prime}, \xi^{\prime \prime}) &
\longleftrightarrow & (\xi^{\prime},x^{\prime
\prime};-x^{\prime},\xi^{\prime \prime}) & \longleftrightarrow &
(x^{\prime},x^{\prime \prime};\xi^{\prime}, \xi^{\prime \prime})
\end{array}
\end{equation}
Under this identification, we can prove that $S_3  \subset T^*(T_MX) \simeq
T_{\Delta_{T_MX}}^*(T_{M\times M}(X\times X)) $ is equal to the normal cone
$S^{\prime}= C_{T^*_MX}(\SS(F)) \subset T_{T^*_MX}(T^*X)$ as follows. In the
associated local coordinates $(x^{\prime}, x^{\prime \prime},t;\xi^{\prime},
\xi^{\prime \prime})$ ($t>0$) of $\Delta_{\Omega_X}
\utimes{\Delta_{\tl{X_M}}} T_{\Delta_{\tl{X_M}}}^*(\tl{(X\times X)_{M\times
M}}) \ (\simeq \Delta_{\Omega_X} \utimes{\Delta_X} T_{\Delta_X}^*(X\times X)
\simeq \Omega_X \utimes{X}T^*X)$, its subset
$\t{j_1}^{-1}\t{\tl{p_1}}(\Delta_{\Omega_X} \utimes{\Delta_X}S)$ is
expressed by
\begin{equation}
\{ (x^{\prime}, x^{\prime \prime},t;\xi^{\prime}, \xi^{\prime \prime}) \in
\Delta_{\Omega_X}
\utimes{\Delta_{\tl{X_M}}}T_{\Delta_{\tl{X_M}}}^*(\tl{(X\times X)_{M\times
M}}) \ | \ (tx^{\prime}, x^{\prime \prime}; t^{-1}\xi^{\prime}, \xi^{\prime
\prime}) \in \SS(F) \}.
\end{equation}
Hence we have
\begin{eqnarray}
\lefteqn{(x^{\prime},x^{\prime \prime};\xi^{\prime},\xi^{\prime \prime}) \in
S_3=S_2 \cap T_{\Delta_{T_MX}}^*(T_{M\times M}(X\times X))} \nonumber \\
&\Longleftrightarrow & \exists (x_n^{\prime}, x_n^{\prime \prime},t_n
;\xi_n^{\prime}, \xi_n^{\prime \prime}) \in \Delta_{\Omega_X}
\utimes{\Delta_{\tl{X_M}}}T_{\Delta_{\tl{X_M}}}^*(\tl{(X\times X)_{M\times
M}}) \nonumber \\
& & \hspace{10mm} \text{s.t.\ } \begin{cases} (x_n^{\prime}, x_n^{\prime
\prime},t_n ;\xi_n^{\prime}, \xi_n^{\prime \prime}) \overset{n \to
\infty}{\longrightarrow} (x^{\prime},x^{\prime
\prime},0;\xi^{\prime},\xi^{\prime \prime}), \\ (t_nx_n^{\prime},
x_n^{\prime \prime}; t_n^{-1}\xi_n^{\prime}, \xi_n^{\prime \prime}) \in
\SS(F) \end{cases}\\
&\Longleftrightarrow & \exists (x_n^{\prime}, x_n^{\prime \prime},t_n
;\xi_n^{\prime}, \xi_n^{\prime \prime}) \in \Delta_{\Omega_X}
\utimes{\Delta_{\tl{X_M}}}T_{\Delta_{\tl{X_M}}}^*(\tl{(X\times X)_{M\times
M}}) \nonumber \\
& & \hspace{10mm} \text{s.t.\ } \begin{cases} (x_n^{\prime}, x_n^{\prime
\prime},t_n ;\xi_n^{\prime}, \xi_n^{\prime \prime}) \overset{n \to
\infty}{\longrightarrow} (x^{\prime},x^{\prime
\prime},0;\xi^{\prime},\xi^{\prime \prime}), \\ (t_nx_n^{\prime},
x_n^{\prime \prime}; \xi_n^{\prime}, t_n\xi_n^{\prime \prime}) \in \SS(F)
\end{cases}
\end{eqnarray}

\begin{eqnarray}
&\Longleftrightarrow & \exists ((\tl{x_n}^{\prime}, \tl{x_n}^{\prime
\prime};\tl{\xi_n}^{\prime},\tl{\xi_n}^{\prime \prime}), c_n) \in \SS(F)
\times \Real_{>0} \nonumber \\
& & \hspace{10mm}\text{s.t.\ }\begin{cases} (\tl{x_n}^{\prime},
\tl{x_n}^{\prime \prime};\tl{\xi_n}^{\prime},\tl{\xi_n}^{\prime \prime})
\overset{n \to \infty}{\longrightarrow} (0,x^{\prime \prime};\xi^{\prime},
0), \\
(c_n\tl{x_n}^{\prime},c_n\tl{\xi_n}^{\prime \prime}) \overset{n \to
\infty}{\longrightarrow} (x^{\prime}, \xi^{\prime \prime}) \end{cases} \\
&\Longleftrightarrow & (x^{\prime},x^{\prime
\prime};\xi^{\prime},\xi^{\prime \prime}) \in S^{\prime}= C_{T^*_MX}(\SS(F))
\subset T_{T^*_MX}T^*X.
\end{eqnarray}
We thus obtained the morphism ${\bf A}$:
\begin{equation}
\RG_{S}(T^*X; \mu_{\Delta_X}(F \boxtimes \D F)) \longrightarrow
\RG_{S^{\prime}}(T^*\Vb ;\mu_{\Delta_{\Vb}} (\nu_{M\times M}(F\boxtimes \D
F))).
\end{equation}
We can construct also the morphism ${\bf B}$ in 
Diagram \ref{prp:6-1}.a as follows.
\begin{eqnarray}
\lefteqn{\RG_S(T^*X; \mu_{\Delta_X}(h_*(\phi^{-1}F \otimes \D F)))}
\nonumber \\
&\longrightarrow& \RG_{S^{\prime}}(T^*\Vb ; \mu_{\Delta_{\Vb}}(\nu_{M \times
M}(h_*(\phi^{-1}F \otimes \D F)))) \\
&\longrightarrow& \RG_{S^{\prime}}(T^*\Vb ;\mu_{\Delta_{\Vb}}(\tl{h}_*
\nu_M(\phi^{-1}F \otimes \D F))),
\end{eqnarray}
where the first morphism is constructed in the same way as ${\bf A}$ and we
used \cite[Proposition 4.2.4]{K-S} to construct the second morphism. This
completes the proof. \qed
\end{proof}

In what follows, we shall identify $\Lb \simeq
\Lb^{\prime}$ with $T^*M$ and describe
$LC(F,\Phi)=LC(\nu_M(F),\Phi^{\prime})$.
Since our result holds for any conic object on any vector bundle over $M$, 
let us consider the following general setting. Let $\tau \colon \Vb 
\longtwoheadrightarrow M$ be a real vector bundle 
of rank $r>0$ over $M$ and $\psi \colon 
\Vb \longrightarrow \Vb$ its endomorphism. Assume 
that the fixed point set of $\psi$ is the zero-section $M$ of $\Vb$. 
This assumption implies that 
\begin{align}
1 \notin \Ev(\psi_x) 
\text{ for any $x \in M$.} \label{cd:6-1}
\end{align}
Suppose that we are given a conic $\Real$-constructible 
object $G \in \Dc(\Vb)$ on $\Vb$ and a morphism $\Psi \colon \psi^{-1}G 
\longrightarrow G$ in $\Dc(\Vb)$. From these data, we can construct the 
Lefschetz bundle $\Lb_0 \simeq T^*M$ associated with $\psi$ and the 
Lefschetz cycle $LC(G,\Psi)$ in it.

\smallskip

Fix a point $\overset{\circ}{x} \in M$ and 
consider the linear homomorphism 
$\psi_{\overset{\circ}{x}} \colon \Vb_{\overset{\circ}{x}}
 \longrightarrow \Vb_{\overset{\circ}{x}}$.
Let $\lambda_1,\dots ,\lambda_d$ be the eigenvalues of 
$\psi_{\overset{\circ}{x}}$ on $\Real_{\ge 1}$ 
and $\lambda_{d+1},\dots ,\lambda_r$ the remaining ones. 
Since these eigenvalues vary 
depending on $x \in M$ continuously, we denote 
their continuous extensions to a 
neighborhood of $\overset{\circ}{x}$ in $M$ by 
$\lambda_1(x),\dots \lambda_r(x)$. 
Then for a sufficiently large 
$R>0$ we have $\lambda_1,\dots,\lambda_d \in [1,R]$. 
Moreover there exists a sufficiently small $\e>0$ such that 
\begin{align}
\lambda_{d+1},\dots,\lambda_r \not\in
\{ z \in \Comp \ | \ \Re z \ge 1, |\Im z| \le \e \}.
\end{align}
By the continuity of the eigenvalues, 
there exists a sufficiently small 
neighborhood $U$ of $\overset{\circ}{x}$ in $M$ such that 
\begin{align}
\lambda_{d+1}(x),\dots, \lambda_r(x) & \not\in 
\{z \in \Comp \ | \ \Re z \ge 1, |\Im z|\le \e \}, \\
\lambda_1(x),\dots, \lambda_d(x) & \in 
\{z \in \Comp \ | \ 1 < \Re z < R+1, 
 |\Im z| < \e \} 
\end{align}
for any $x \in U$.
If necessary, replacing 
$U$ by a smaller one, we may assume also that $\Vb$ 
is trivial on $U$. For $x \in U$ we set 
\begin{align}
P_x= \frac{1}{2 \pi i} 
\int_{C}(z- \psi_x)^{-1} dz,
\end{align}
where $C$ is the path on 
the boundary of the set 
$\{z \in \Comp \ | \ 1 < 
 \Re z < R+1, |\Im z| < \e \} 
\subset \Comp$. 
Then $P_x \colon \Vb_x \longrightarrow \Vb_x$ is 
the projector onto the direct sum of 
the generalized eigenspaces associated with the eigenvalues in 
$\{z \in \Comp \ | \ 1 < \Re z <  R+1, 
|\Im z| < \e \} \subset \Comp$. 
The family $\{P_x \}_{x \in U}$ 
defines an endomorphism 
$P$ of $\Vb|_U$, whose image $\MW \subset \Vb|_U$ 
is a subbundle of $\Vb|_U$.

\begin{dfn}\label{dfn:6-1}
We call $\MW:=\Im P \subset \Vb|_U$  
the minimal expanding subbundle of $\Vb|_U$ 
(on the neighborhood $U$ of $\overset{\circ}{x} \in M$).
\end{dfn}

\begin{dfn}\label{dfn:6-2}{\bf (\cite[Section 9.6]{K-S})} 
We say that a subbundle $\ME$ of $\Vb|_U$ is an expanding 
subbundle if it satisfies the following conditions:
\begin{enumerate}
\item $\psi|_{\tau^{-1}(U)}(\ME) \subset \ME$.
\item $\MW$ is a subbundle of $\ME$.
\item $\ME_x^\Comp \subset 
\bigoplus_{\lambda \not\in [0,1]}(\Vb_x)_\lambda^\Comp$ 
for any $x \in U$.
\end{enumerate}
\end{dfn}

For any expanding subbundle $\ME$ of $\Vb|_U$, the induced morphism 
$\psi|_{\ME} \colon \ME \longrightarrow \ME$ is an 
isomorphism of vector bundles.

\begin{dfn}\label{dfn:6-3}(\cite{Braden}) 
Let $\tau_{\ME} \colon \ME \longrightarrow U$ be an 
expanding subbundle of $\Vb|_U$ and 
$i_\ME \colon U \longhookrightarrow \ME$ its zero-section.
We define an object $G_\ME^{!-1} \in \Dc(U)$ by 
\begin{equation}
G_\ME^{!-1}:=i_\ME^!(G|_\ME) \simeq R{\tau_\ME}_!(G|_\ME)
\end{equation}
and its endomorphism $\Psi_\ME^{!-1}: 
G_\ME^{!-1} \longrightarrow G_\ME^{!-1}$ by 
the composite of the morphisms 
\begin{eqnarray}
R{\tau_\ME}_!(G|_\ME) 
& \longrightarrow & 
R{\tau_\ME}_!\tl{\psi}_*\tl{\psi}^{-1}(G|_\ME) \\
& \simeq & 
R{\tau_\ME}_!\tl{\psi}_!((\psi^{-1}G)|_\ME) 
\simeq R{\tau_\ME}_! ((\psi^{-1}G)|_\ME) 
 \\
& \overset{\Psi}{\longrightarrow} &
R{\tau_\ME}_!(G|_\ME).
\end{eqnarray}
Here we set $\tl{\psi}:=\psi|_{\ME}$ and the 
first morphism above is induced by 
the adjunction.
We call the pair $(G_\ME^{!-1},\Psi_\ME^{!-1})$ 
the hyperbolic localization 
of $(G,\Psi)$ with respect to $\ME$.
\end{dfn}

\begin{prp}\label{prp:6-3}
Let $\overset{\circ}{x} \in M$ 
be a point of $M$. Then there exists a 
sufficiently small open neighborhood $U$ of $\overset{\circ}{x}$ 
in $M$ such that for any compact subanalytic subset $K$ of $U$ and 
for any expanding subbundle $\ME$ of $\Vb|_U$, we have
\begin{align}
\int_{\tl{U}}C(G_{\tl{K}}|_{\tl{U}},\Psi_{\tl{K}}|_{\tl{U}})
=\tr((G_{\tl{K}})_\ME^{!-1},(\Psi_{\tl{K}})_\ME^{!-1}).
\end{align}
Here we set $\tl{U}:=\tau^{-1}(U)$ 
and $\tl{K}:=\tau^{-1}(K)$.
\end{prp}

The proof of this proposition is completely 
similar to that of \cite[Proposition 9.6.12]{K-S} 
and we omit it here. 
By the isomorphism 
$(G_{\tl{K}})_\ME^{!-1} \simeq (G_\ME^{!-1})_K$, 
we thus obtain the equality
\begin{align}\label{ADDEE} 
\int_{\tl{U}}C(G_{\tl{K}}|_{\tl{U}},\Psi_{\tl{K}}|_{\tl{U}})
=\tr((G_\ME^{!-1})_K,(\Psi_\ME^{!-1})_K).
\end{align}

Take a sufficiently small open subset $U$ of $M$ 
for which Proposition \ref{prp:6-3} holds 
and define a constructible 
function $\varphi(G_\ME^{!-1},\Psi_\ME^{!-1})$ on it 
associated to the hyperbolic localization 
$(G_\ME^{!-1},\Psi_\ME^{!-1})$ by 
\begin{equation}
\varphi(G_\ME^{!-1},\Psi_\ME^{!-1})
(x):=\dsum_{j \in \Z} (-1)^j\tr\{ H^j (G_\ME^{!-1})_x
\overset{ \Psi_\ME^{!-1}
|_{\{x\}}}{\longrightarrow} H^j(G_\ME^{!-1} 
)_x\}. 
\end{equation}
Then by applying \eqref{ADDEE} to 
the special case where $K$ is a point, we find that 
it does not depend on the choice of 
the expanding subbundle $\ME$. 
Hence we can glue such locally defined 
constructible functions to obtain a global 
one $\varphi_M(G,\Psi)$ on $M$.

\begin{prp}\label{prp:6-2}
Under the condition \eqref{cd:6-1} we have the equality 
\begin{align}\label{LEQ} 
LC(G,\Psi)=CC(\varphi_M(G,\Psi))
\end{align}
as Lagrangian cycles in $T^*M$.
\end{prp}

\begin{proof}
Let $\pi_M \colon T^*M \longrightarrow M$ be the projection. 
Fix a point $\overset{\circ}{x} \in M$ and compare 
the both sides of \eqref{LEQ} on a neighborhood of 
$\pi_M^{-1}(\overset{\circ}{x})  \subset T^*M$. 
By the homotopy invariance of Lefschetz cycles 
(see Proposition \ref{prp:4-9}), 
\cite[Proposition 9.6.8]{K-S} and \eqref{ADDEE}, 
taking a sufficiently small open neighborhood $U$ 
of $\overset{\circ}{x}$ and replacing $\psi$ 
with $t\psi$ for $|1-t| \ll 1$, we may assume 
the following conditions:
\begin{itemize}
 \item[(1)] $\Vb|_U$ is trivial.
 \item[(2)] The open subset $U$ satisfies the 
condition of Proposition \ref{prp:6-3}.
 \item[(3)] $\Ev( \psi_x) \cap \{z \in 
\Comp \ | \ |z|=1\}=\emptyset$ for any $x \in U$.
\end{itemize}
It is enough to show that 
\begin{align}
LC(G|_{\tl{U}},\Psi|_{\tl{U}})
=CC(\varphi_M(G,\Psi)|_U),
\end{align}
where $\tl{U}=\tau^{-1}(U)$. 
As in the proof of \cite[Proposition 9.6.12]{K-S} 
we can construct subbundles 
$\Vb_+$ and $\Vb_-$ of $\Vb|_U$ for which 
we have the direct sum decomposition 
$\Vb|_U =\Vb_+ \oplus \Vb_-$ 
and a metric on $\Vb|_U$ such that
\begin{align}
 & \text{there exist constants } c_1, 
c_2 \text{ with } 0<c_1<1<c_2 \text{ satisfying the 
condition } \notag \\
 & | \psi_{x}(v_-)| \le c_1|v_-| \, 
(v_- \in \Vb_{-,x}), \; | \psi_{x}(v_+)| \ge c_2|v_+| 
\, (v_+ \in \Vb_{+,x}) \text{ for any } x \in U .
\end{align} 
By using this metric, we set 
\begin{align}
Z:=\{ (x,v_+,v_-) \in \Vb|_U \ | \ |v_+| < a, |v_-| \le b \},
\end{align}
for some fixed constants $a,b>0$.
Then $\psi^{-1}(Z) \cap Z$ is open 
in $Z$ and closed in $\psi^{-1}(Z)$ and hence 
we can construct a morphism
\begin{align}
\Psi_Z \colon \psi^{-1}(G_Z) \longrightarrow G_Z
\end{align}
induced by $\Psi \colon \psi^{-1}G \longrightarrow G$. 
Since $\Vb_+$ is an expanding subbundle of $\Vb|_U$, we have 
\begin{align}
\varphi_M(G,\Psi)|_U
=\varphi(G_{\Vb_+}^{!-1},\Psi_{\Vb_+}^{!-1}). 
\end{align}
Moreover we have 
\begin{align}
CC(\varphi(G_{\Vb_+}^{!-1},\Psi_{\Vb_+}^{!-1}))
=LC(G_{\Vb_+}^{!-1},\Psi_{\Vb_+}^{!-1}),
\end{align}
by Proposition \ref{thm:5-8}.
Thus we have to show that 
\begin{align}
LC(G|_{\tl{U}},\Psi|_{\tl{U}})
=LC(G_{\Vb_+}^{!-1},\Psi_{\Vb_+}^{!-1})
\end{align}
as Lagrangian cycles in $T^*U$.
In what follows, for simplicity we write $M$ instead of $U$ and 
$\Vb|_U=\tl{U}, G|_{\tl{U}}, \Psi|_{\tl{U}}$ etc.\ by 
$\Vb, G, \Psi$ etc.\ respectively.

Let us take a $\mu$-stratification 
$\Vb= \bigsqcup_{\alpha \in A} \Vb_{\alpha}$ of $\Vb$ 
which satisfies the following three conditions.
\begin{enumerate}
\item There exists a subset $B \subset A$ such 
that the zero-section $M 
\subset \Vb$ of $\Vb$ is $\bigsqcup_{\beta \in B} 
\Vb_{\beta}$.
\item $\SS(G_Z) \subset \bigsqcup_{\alpha \in A} 
T^*_{\Vb_{\alpha}}\Vb$ in 
$T^*\Vb$.
\item $\SS(G_{\Vb_+}^{!-1}), \SS(R\tau_*G_Z) \subset 
\bigsqcup_{\beta \in B}T^*_{\Vb_{\beta}}M$ in 
$T^*M$.
\end{enumerate}
For $\beta \in B$, we shall denote 
$\Vb_{\beta} \subset M$ by $M_{\beta}$. 
Namely $M= \bigsqcup_{\beta \in B}M_{\beta}$ 
is a $\mu$-stratification of 
$M$. Set $\Lambda= \bigsqcup_{\beta \in B}
T^*_{M_{\beta}}M \subset T^*M$. By 
the conditions above, we obtain
\begin{align}
\supp(LC(G,\Psi)), \hspace{3mm}\supp(
LC(G_{\Vb_+}^{!-1},\Psi_{\Vb_+}^{!-1})) \subset \Lambda.
\end{align}
Therefore it suffices to show that $LC(G,\Psi)$ coincides with 
$LC(G_{\Vb_+}^{!-1},\Psi_{\Vb_+}^{!-1})$ on an 
open dense subset of $\Lambda$. Let $\Lambda_0$ be 
an open dense smooth subanalytic subset of 
$\Lambda$ whose decomposition 
$\Lambda_0= \bigsqcup_{i \in I}\Lambda_i$ into 
connected components 
satisfies the condition
\begin{align}
``\text{For any $i \in I$, there exists 
$\beta_i \in B$ such that $\Lambda_i 
\subset T_{M_{\beta_i}}^*M$.}"
\end{align}
Let us fix $\Lambda_i$ and $M_{\beta_i}$ 
as above and compare $LC(G,\Psi)$ 
with $LC(G_{\Vb_+}^{!-1},\Psi_{\Vb_+}^{!-1})$ on $\Lambda_i$.
 Take a point $p_0 \in \Lambda_i$ and 
set $x_0 =\pi_M(p_0) \in M_{\beta_i}$. Let $f \colon M 
\longrightarrow \Real$ be a real analytic function 
(defined in an open neighborhood of $x_0$) which satisfies 
that $p_0=(x_0;df(x_0)) \in \Lambda_i$, $f(x_0)=0$ and the 
Hessian ${\rm 
Hess}(f|_{M_{\beta_i}})$ is positive definite. Then 
by Corollary \ref{cor:5-5}, we have
\begin{align}
LC(G_{\Vb_+}^{!-1},\Psi_{\Vb_+}^{!-1})=m_i 
\cdot [T_{M_{\beta_i}}^*M]
\end{align}
in an open neighborhood of $\Lambda_i$ in 
$T^*M$, where $m_i \in \Comp$ is 
defined by
\begin{align}
m_i:= \dsum_{j \in \Z} (-1)^j \tr\{ H^j_{\{f \geq 
0\}}(B(x_0,\delta);G_{\Vb_+}^{!-1}) 
\overset{\Psi_{\Vb_+}^{!-1}}{\longrightarrow} 
H^j_{\{f \geq 0\}}(B(x_0,\delta);G_{\Vb_+}^{!-1})\}
\end{align}
for sufficiently small $\delta >0$. Set 
$V:=B(x_0,\delta)$ and $W:=V \cap \{ 
f<0\}$ in $M$. Then we have
\begin{align}\label{eq:6-16}
m_i=\tr(\RG_V(G_{\Vb_+}^{!-1})),\RG_V(\Psi_{\Vb_+}^{!-1})))
-\tr(\RG_W(G_{\Vb_+}^{!-1})),\RG_W(\Psi_{\Vb_+}^{!-1}))).
\end{align}
Set also $\tl{V}:=\tau^{-1}(V)$, $\tl{W}:=\tau^{-1}(W)\subset \Vb$ and 
$\tl{f}:=f \circ \tau \colon \Vb \longrightarrow \Real$. 
Since we work in a sufficiently small open neighborhood of $x_0$, 
we may assume $M=\Real^m, x_0=0$.
Set $g(x):=|x|^2=x_1^2+\dots+x_m^2$.

\begin{lem}\label{lem:6-7}
For a sufficiently small $\delta>0$, we have 
\begin{align}
\tr(\RG_V(G_{\Vb_+}^{!-1}),\RG_V(\Psi_{\Vb_+}^{!-1}))
&=
\tr(\RG_{\tl{V}}(G_Z),\RG_{\tl{V}}(\Psi_Z)), \label{eq:6-5}\\
\tr(\RG_W(G_{\Vb_+}^{!-1}),\RG_W(\Psi_{\Vb_+}^{!-1}))
&=
\tr(\RG_{\tl{W}}(G_Z),\RG_{\tl{W}}(\Psi_Z)).\label{eq:6-6}
\end{align}
\end{lem}

\begin{proof}
By the microlocal Bertini-Sard theorem 
(\cite[Proposition 8.3.12]{K-S}),  
there exist $\delta_0, \e_0>0$ such that we have 
\begin{align}
\Lambda \cap \Real_{\ge 0} \Lambda_g 
\cap \pi_M^{-1}(\{0 <|x| \le \delta_0 \}) 
\subset T^*_MM, \label{cd:6-3}\\
(\Lambda +T^*_{\{f=0\}}M) \cap \Real_{\ge 0} \Lambda_g 
\cap \pi_M^{-1}(\{0 <|x| \le \delta_0 \}) 
\subset T^*_MM, \label{cd:6-4}\\
\Lambda \cap \Real_{\ge 0} \Lambda_f 
\cap \pi_M^{-1}(\{-\e_0 \le f<0 \}) \subset T^*_MM. \label{cd:6-5}
\end{align}
Fix a constant $\delta$ such that $0<\delta<\delta_0$ and set 
$S:=\{x \in M \ | \ f(x)=0, |x|=\delta \}$.
For $x \in S$ consider covectors 
$\xi \in \Real_{\ge 0}\Lambda_f \cap T^*_xM, 
\eta \in \Real_{\ge 0}\Lambda_g \cap  T^*_xM$.
Assume that $\xi+\eta=\lambda \in \Lambda$.
Then we have 
$\lambda-\xi=\eta \in (\Lambda +T^*_{\{f=0\}}M) 
\cap \Real_{\ge 0}\Lambda_g$
and thus obtain $\eta=0$ by \eqref{cd:6-4}.
This implies that we have 
 $\xi=\lambda \in \Lambda \cap \Real_{\ge 0}\Lambda_f 
\cap T^*_xM$ and $\xi=\lambda=0$.
By this argument we obtain
\begin{align}
\Lambda \cap 
(\Real_{\ge 0}\Lambda_f+\Real_{\ge 0}\Lambda_g)
\cap T^*_xM
\subset \{0\}
\end{align}
for any $x \in S$.
By the compactness of $S$, there exists an open neighborhood $O$ 
of $S$ such that for any $x \in O$ we have 
\begin{align}
\SS(G_{\Vb_+}^{!-1}) \cap 
(\Real_{\ge 0}\Lambda_f+\Real_{\ge 0}\Lambda_g)
\cap T^*_xM
\subset \{0\}, \label{cd:6-6}\\
\SS(R\tau_*G_Z) \cap 
(\Real_{\ge 0}\Lambda_f+\Real_{\ge 0}\Lambda_g)
\cap T^*_xM
\subset \{0\}.\label{cd:6-7}
\end{align}

First, let us prove the equality \eqref{eq:6-5}. 
By the microlocal Morse lemma, we have
\begin{align}
\RG(B(x_0,\delta);G_{\Vb_+}^{!-1})
&\simeq 
\RG(\overline{B(x_0,\delta)};G_{\Vb_+}^{!-1}),\\
\RG(\tau^{-1}(B(x_0,\delta));G_Z) 
&\simeq
\RG(\tau^{-1}(\overline{B(x_0,\delta)});G_Z)
\end{align}
for $0<\delta<\delta_0$. Thus for $K:=\overline{B(x_0,\delta)}$ 
we obtain 
\begin{align}
\tr(\RG_V(G_{\Vb_+}^{!-1}),\RG_V(\Psi_{\Vb_+}^{!-1}))
&=\tr((G_{\Vb_+}^{!-1})_K,(\Psi_{\Vb_+}^{!-1})_K), \\
\tr(\RG_{\tl{V}}(G_Z),\RG_{\tl{V}}(\Psi_Z))
&=\tr((G_Z)_{\tl{K}},(\Psi_Z)_{\tl{K}}).
\end{align}
Moreover by Proposition \ref{prp:6-3} and the local invariance of 
characteristic classes, we have the equality
\begin{align}
\tr((G_{\Vb_+}^{!-1})_K,(\Psi_{\Vb_+}^{!-1})_K)
&=
\int_\Vb C(G_{\tl{K}},\Psi_{\tl{K}}) \\
&=
\int_\Vb C((G_Z)_{\tl{K}},(\Psi_Z)_{\tl{K}}) \\
&=
\tr((G_Z)_{\tl{K}},(\Psi_Z)_{\tl{K}}).
\end{align}
We thus obtain the equality \eqref{eq:6-5}.

Next let us prove the equality \eqref{eq:6-6}. 
For $\e >0$ we define an open subset $U_\e$ of $W$ by
\begin{align}
U_\e:=\{x \in W \ | \ {\rm dist}(x,M \setminus W)>\e\}.
\end{align}
For sufficiently small $\e>0$, 
outer conormal vectors of $\partial U_\e$ are contained in 
$\Real_{\ge 0}\Lambda_f+\Real_{\ge 0}\Lambda_g 
\cap \pi_M^{-1}(O)$. 
Thus by the conditions \eqref{cd:6-3}, \eqref{cd:6-5}, 
\eqref{cd:6-6} and \eqref{cd:6-7} 
we can apply the non-characteristic deformation lemma 
to $G_{\Vb_+}^{!-1}, R\tau_*G_Z$ and 
the family $\{U_\e\}_\e$
to obtain 
\begin{align}
\RG(W; G_{\Vb_+}^{!-1})
& \simto
\RG(\overline{U_\e};G_{\Vb_+}^{!-1}), \\
\RG(W;R\tau_*G_Z)
& \simto
\RG(\overline{U_\e};R\tau_*G_Z)
\end{align}
for sufficiently small $\e>0$.
Replacing $K$ by $\overline{U_\e}$ in 
the proof of \eqref{eq:6-5}, we 
obtain the equality \eqref{eq:6-6}. 
\qed
\end{proof}

Applying Lemma \ref{lem:6-7} and Theorem \ref{thm:5-1} to the pair 
$(\RG_{\tl{V}}(G_Z),\RG_{\tl{V}}(\Psi_Z))$, we obtain
\begin{align}\label{eq:6-23}
\tr(\RG_V(G_{\Vb_+}^{!-1}),\RG_V(\Psi_{\Vb_+}^{!-1}))
=\sharp ([\sigma_{f}] \cap 
LC(\RG_{\tl{V}}(G_Z), \RG_{\tl{V}}(\Psi_Z))).
\end{align}
Now define a real analytic function (defined on a neighborhood 
of $\tau^{-1}(x_0) \subset \Vb$) 
$\tl{g}\colon \Vb \longrightarrow \Real$ by $\tl{g}:=g \circ \tau$.
Then by the microlocal Bertini-Sard theorem, there exists 
$\delta_1>0$ such that 
\begin{align}
\SS(G_Z) \cap \Lambda_{\tl{g}} \cap 
\pi_{\Vb}^{-1}(\{ v \in \Vb \ | \ 0<|\tau(v)| \le \delta_1 \})
=\emptyset, \label{cd:6-8}
\end{align}
where $\pi_{\Vb} \colon T^*\Vb \longrightarrow \Vb$ is the projection.
Moreover by the proof of \cite[Theorem 9.5.6]{K-S},
there exists $\delta_2>0$ such that 
\begin{align}
 c\ge0, 0<|x|\le \delta_2, f(x)>0 \Longrightarrow (x;c \cdot dg(x)+df(x)) 
 \not\in \Lambda. \label{cd:6-9}
\end{align}
Replacing the constant $\delta$ by a smaller one, we may assume that 
$0<\delta<\min(\delta_0,\delta_1,\delta_2)$.
By the condition (i), \eqref{cd:6-8} and the 
definition of $\Lambda$ we have
\begin{eqnarray}
\supp(LC(\RG_{\tl{V}}(G_Z),\RG_{\tl{V}}(\Psi_Z)))
&\subset & \SS(\RG_{\tl{V}}(G_Z)) \cap \Lb_0 
\label{eq:6-28}\\
&\subset & \{ \SS(G_Z) \cup (\SS(G_Z) +
T^*_{\partial \tl{V}}\Vb)\} \cap \Lb_0 \\
&\subset & \Lambda \cup (\Lambda +
T^*_{\partial V}M)=:\Lambda^{\prime}.\label{eq:6-30}
\end{eqnarray}
Since $\Lambda^{\prime}$ is isotropic, 
by the microlocal Bertini-Sard 
theorem there exists sufficiently small 
$\e_1 >0$ such that
\begin{align}
\Lambda^{\prime} \cap \Lambda_{f} \cap \pi_M^{-1}(\{0 < | f |\le \e_1
\})=\emptyset. \label{cd:6-10}
\end{align}
Arguing as in the proof of \cite[Theorem 9.5.6]{K-S} 
by using the conditions \eqref{cd:6-4}, \eqref{cd:6-9} and 
\eqref{cd:6-10} and the estimate 
\eqref{eq:6-28}-\eqref{eq:6-30}, we obtain
\begin{align}
\Lambda_{f} \cap \supp(LC(\RG_{\tl{V}}(G_Z), \RG_{\tl{V}}(\Psi_Z))) 
\subset \pi_M^{-1}(\{f<-\e_1 \}) \sqcup \{p_0\}.
\end{align}
Hence from \eqref{eq:6-23} we deduce
\begin{align}\label{eq:6-38}
& \lefteqn{\tr(\RG_V(G_{\Vb_+}^{!-1}),
 \RG_V(\Psi_{\Vb_+}^{!-1}))}  \notag \\
&=\sharp \{\pi_M^{-1} (\{f <-\e_1 \}) \cap [\sigma_{f}] \cap 
LC(\RG_{\tl{V}}(G_Z), \RG_{\tl{V}}(\Psi_Z))\} +[\sigma_{f}] 
\underset{p_0}{\cdot} LC(G_Z,\Psi_Z),
\end{align}
where $[\sigma_{f}] \underset{p_0}{\cdot} 
LC(G_Z,\Psi_Z)$ is the local 
intersection number of $[\sigma_{f}]$ and 
$LC(G_Z,\Psi_Z)$ at $p_0 \in 
\Lambda_i$. 

The other term 
$\tr(\RG_W(G_{\Vb_+}^{!-1}),\RG_W(\Psi_{\Vb_+}^{!-1}))
=\tr(\RG_{\tl{W}}(G_Z),\RG_{\tl{W}}(\Psi_Z))$ 
can be calculated as follows. For $\e>0$, 
set $W_{\e}:=W \cap \{ f <-\e\}$ and 
$\tl{W_{\e}}:=\tl{W} \cap \{ \tl{f}<-\e\}=\tau^{-1}(W_{\e})$.

\begin{lem}\label{lem:6-3}
There exists sufficiently small $\e_2>0$ such that
\begin{align}
\tr(\RG_{\tl{W}}(G_Z),\RG_{\tl{W}}(\Psi_Z))
=\tr(\RG_{\tl{W_{\e}}}(G_Z),\RG_{\tl{W_{\e}}}(\Psi_Z))
\end{align}
for any $0<\e< \e_2$. 
\end{lem}

\begin{proof}
Set $\Sigma:=\SS(\RG_{\tl{V}}(G_Z)) \subset T^*\Vb$. 
Then by the microlocal 
Bertini-Sard theorem there exists $\e_2>0$ 
such that
\begin{align}
\Sigma \cap \Lambda_{\tl{f}} \cap \pi^{-1}(\{-\e_2 
\le \tl{f} <0\})=\emptyset.
\end{align}
Hence by the microlocal Morse lemma 
(\cite[Corollary 5.4.19]{K-S}), 
for $0<\e<\e_2$ we obtain
\begin{align}
\RG(\{\tl{f}<0\};\RG_{\tl{V}}(G_Z)) \simto 
\RG(\{\tl{f}<-\e\};\RG_{\tl{V}}(G_Z)).
\end{align}
\qed
\end{proof}

Let us continue the proof of Proposition \ref{prp:6-2}. 
By Lemma \ref{lem:6-7},
Lemma \ref{lem:6-3} and Theorem \ref{thm:5-1}, we obtain
\begin{align}\label{eq:6-27}
\tr(\RG_W(G_{\Vb_+}^{!-1}),\RG_W(\Psi_{\Vb_+}^{!-1}))
=\sharp ([\sigma_{f}] \cap 
LC(\RG_{\tl{W_{\e}}}(G_Z),\RG_{\tl{W_{\e}}}(\Psi_Z)))
\end{align}
for $0<\e<\e_2$. 
Moreover it follows from the condition (i) 
and the definition of $\Lambda$ that
\begin{eqnarray}
\supp(LC(\RG_{\tl{W_{\e}}}(G_Z),\RG_{\tl{W_{\e}}}(\Psi_Z)))
&\subset & \SS(\RG_{\{\tl{f}<-\e\}}(\RG_{\tl{V}}(G_Z))) \cap \Lb_0 \\
&\subset & \Lambda^{\prime} +\Real_{\leq 0}\Lambda_{f}.
\end{eqnarray}
Comparing this last estimate with \eqref{cd:6-10}, 
we obtain
\begin{align}
\Lambda_{f} \cap 
\supp(LC(\RG_{\tl{W_{\e}}}(G_Z),\RG_{\tl{W_{\e}}}(\Psi_Z))) \subset 
\pi_M^{-1}(\{f<-\e_1 \})
\end{align}
for $0<\e<\min(\e_1,\e_2)$. Since
\begin{align}
LC(\RG_{\tl{W_{\e}}}(G), 
\RG_{\tl{W_{\e}}}(\Psi))=LC(\RG_{\tl{V}}(G),\RG_{\tl{V}}(\Psi))
\end{align}
on $\pi_M^{-1}(\{f<-\e_1 \})$, from \eqref{eq:6-27} we obtain
\begin{align}\label{eq:6-36}
\lefteqn{\tr(\RG_W(G_{\Vb_+}^{!-1}),\RG_W(\Psi_{\Vb_+}^{!-1}))} \notag\\
&=\sharp\{\pi_M^{-1}(\{f <-\e_1 \})\cap [\sigma_{f}] 
\cap LC(\RG_{\tl{V}}(G_Z), \RG_{\tl{V}}(\Psi_Z))\}.
\end{align}
Putting \eqref{eq:6-38} and \eqref{eq:6-36} 
into \eqref{eq:6-16}, we finally 
obtain
\begin{align}
m_i =[\sigma_{f}] \underset{p_0}{\cdot} LC(G_Z,\Psi_Z),
\end{align}
which shows
\begin{align}
LC(G_Z,\Psi_Z)=LC(G_{\Vb_+}^{!-1},\Psi_{\Vb_+}^{!-1}) \label{eq:6-40}
\end{align}
on $\Lambda_i$. 
By the local invariance of Lefschetz cycles, we have
\begin{align}
LC(G_Z,\Psi_Z)
=LC(G,\Psi).\label{eq:6-41}
\end{align}
By combining \eqref{eq:6-40} with \eqref{eq:6-41} 
we obtain the assertion. 
\qed
\end{proof}

 Now we return to the situation at the 
beginning of this section. Namely for a fixed 
point component $M$ of $\phi$ 
assume the conditions $\supp (F) \cap M 
\subset M_{\reg}$ and  
\begin{equation}
``\text{$1\notin \Ev(\phi'_x)$ for any $x \in \supp(F) \cap M$.}"
\end{equation}

\begin{dfn}
We define a $\Comp$-valued constructible function 
$\theta (F,\Phi)_M \in \CF(M_{\reg})_{\Comp}$ 
on $M_{\reg}$ by 
\begin{equation}
\theta (F,\Phi)_M = \varphi_{M_{\reg}}
(\nu_{M_{\reg}}(F),\Phi'). 
\end{equation}
We call it the local trace function of 
$(F, \Phi )$ on the fixed point component $M$. 
\end{dfn}

By Propositions \ref{prp:6-1} and 
\ref{prp:6-2} we obtain the following explicit 
description of the Lefschetz cycle 
$LC(F,\Phi)_M$.  

\begin{thm}\label{thm:6-4}
We have the equality
\begin{equation}
LC(F,\Phi)_M=CC( \theta (F,\Phi)_M )
\end{equation}
as Lagrangian cycles in $T^*M_{\reg}$. 
In particular, if moreover 
$\supp(F) \cap M$ is compact we have
\begin{equation}
c(F,\Phi)_M =  \dint_{M_{\reg}} \theta (F,\Phi)_M. 
\end{equation}
\end{thm}

\begin{cor}\label{cor:6-6}
Let $X$, $\phi$ and $M$ be as above and $F_1 
\overset{\alpha}{\longrightarrow} F_2 
\overset{\beta}{\longrightarrow} F_3 
\overset{\gamma}{\longrightarrow} F_1[1]$ a 
distinguished triangle in $\Dc(X)$. 
Assume that we are given a morphism of distinguished triangles
\begin{align}
\xymatrix@R=8mm@C=10mm{
\phi^{-1}F_1 \ar[r]^{\phi^{-1}\alpha} \ar[d]^{\Phi_1} & \phi^{-1}F_2 
\ar[r]^{\phi^{-1}\beta} \ar[d]^{\Phi_2} & \phi^{-1}F_3 
\ar[r]^{\phi^{-1}\gamma} \ar[d]^{\Phi_3} & \phi^{-1}F_1[1] 
\ar[d]^{\Phi_1[1]}\\
F_1 \ar[r]^{\alpha} & F_2 \ar[r]^{\beta} &F_3 \ar[r]^{\gamma} & F_1[1] }
\end{align}
in $\Dc(X)$. Then we have
\begin{align}
LC(F_2,\Phi_2)_M=LC(F_1,\Phi_1)_M +LC(F_3,\Phi_3)_M.
\end{align}
\end{cor}

In the complex case, we have the following 
stronger result. 

\begin{thm}\label{thm:6-5}
In the situation of Theorem \ref{thm:6-4}, 
assume moreover that $X$ and $\phi \colon X 
\longrightarrow X$ are complex analytic and $F \in \Db_c(X)$ i.e. $F$ is 
$\Comp$-constructible. Then we have
\begin{equation}
LC(F,\Phi)_M =LC(F|_{M_{\reg}},\Phi|_{M_{\reg}})
=CC(\varphi(F|_{M_{\reg}}, \Phi|_{M_{\reg}}))
\end{equation}
globally on $T^*M_{\reg}$.
\end{thm}

\begin{proof}
By Proposition \ref{prp:6-1}, we have only to prove
\begin{equation}\label{eq:6-43}
LC(\nu_{M_{\reg}}(F),\Phi^{\prime})
=LC(F|_{M_{\reg}},\Phi|_{M_{\reg}}).
\end{equation}
Since these cycles are considered as sections 
of the sheaf of $\L_{M_{\reg}}$ of 
Lagrangian cycles on $T^*M_{\reg}$, it suffices 
to prove \eqref{eq:6-43} locally. 
Namely, for each $x_0 \in M_{\reg}$ we have 
only to prove \eqref{eq:6-43} on an 
open neighborhood of $\pi_{M_{\reg}}^{-1}(x_0)$ 
in $\Lb \simeq T^*M_{\reg}$. This local 
statement can be proved along the same 
line as the proof of Proposition 
\ref{prp:6-2}. Since $\nu_{M_{\reg}}(F)$ 
admits the action of 
the multiplicative group 
$\Comp^{\times} = \Comp \setminus \{ 0 \}$ in the 
complex case, we may use the arguments 
in the proof of \cite[Corollary 
9.6.16]{K-S} for this purpose. This completes the proof.
\qed
\end{proof}

By this theorem we can drop the very technical condition 
on $\supp(F) \cap M$ in Theorem \ref{thm:3-5}. 

\smallskip 
We give a stalk formula of 
$\nu_{M_{\reg}}(F)_\ME^{!-1}$ which is useful to 
calculate the value of the constructible function 
$\theta (F,\Phi)_M=\varphi_{M_{\reg}}
(\nu_{M_{\reg}} (F),\Phi'))$.

\begin{prp}\label{prp:6-7}
In the situation of Theorem \ref{thm:6-4}, 
let $\ME$ be an expanding subbundle 
of $T_MX$ on a neighborhood of $x \in 
\supp (F) \cap M \subset M_{\reg}$. 
Then we have
\begin{align}
 H^k\left(
(\nu_{M_{\reg}} (F)_\ME^{!-1})
\right)_x
\simeq
\varinjlim_{B,Z}H^k_{B \cap Z}(B;F)
\end{align}
for any $k \in \Z$, 
where $B$ ranges through the family 
of open neighborhoods of $x$ in $X$ 
and $Z$ through that of closed subsets of $X$ such that
$C_{M_{\reg}} (Z) \cap 
(\ME_x \setminus \{x\})=\emptyset$.
\end{prp}

\begin{proof}
By the conicness of $\nu_M(F)$, we have
\begin{align}
 H^k\left(
(\nu_M(F)_\ME^{!-1})
\right)_x
& \simeq
H^k\left(
i_\ME^!(\nu_M(F)|_\ME)
\right)_x \\
& \simeq
H^k_{\{x\}}(\ME_x;\nu_M(F)|_{\ME_x}).
\end{align}
In what follows, we use the notation in \eqref{diag:2-1}.
Let $B$ and $Z$ be as in the statement.
Since $A:=C_M(Z)$ is a closed conic subset of $T_MX$,
as in the proof of \cite[Theorem 4.2.3 (iii)]{K-S}, we obtain the chain of morphisms:
\begin{align}
 \RG_{B \cap Z}(B;F)
& \longrightarrow
\RG_{p^{-1}(B \cap Z)}(p^{-1}(B);p^{-1}F) \\
& \longrightarrow
\RG_{p^{-1}(B \cap Z)\cap\Omega_X}(p^{-1}(B) \cap \Omega_X;p^{-1}F) \\
& \longrightarrow
\RG_{(p^{-1}(B \cap Z)\cap\Omega_X)\cup A}
(p^{-1}(B) \cap \Omega_X;Rj_*j^{-1}p^{-1}F) \\
& \longrightarrow
\RG_{\tau^{-1}(B \cap M) \cap A}(\tau^{-1}(B \cap M);\nu_M(F)).
\end{align}
Here we used the fact that $(p^{-1}(B \cap Z)\cap\Omega_X)\cup A$
is closed in $p^{-1}(B)$. 
Since $\ME_x \cap A \subset \{x\}$, the restriction morphism induces the one:
\begin{align}
\RG_{\tau^{-1}(B \cap M) \cap A}(\tau^{-1}(B \cap M);\nu_M(F))
\longrightarrow
\RG_{\{x\}}(\ME_x;\nu_M(F)|_{\ME_x}).
\end{align}
Combining the above morphisms, we obtain the morphism:
\begin{align}
\varinjlim_{B,Z}H^k_{B \cap Z}(B;F) \longrightarrow 
H^k_{\{x\}}(\ME_x;\nu_M(F)|_{\ME_x}). \label{mor:6-1}
\end{align}
Let us show that this is an isomorphism. 
The proof is similar to that of \cite[Theorem 4.2.3 (iii)]{K-S}.
Set $U=B \setminus Z$ and $V:=\tau^{-1}(B \cap M) \setminus C_M(Z)$.
Then $V$ is an conic open subset of $T_MX$ and satisfies 
$\overset{\circ}{\ME_x} \subset V$, 
where we set $\overset{\circ}{\ME_x}:=\ME_x \setminus \{x\}$.
Composing the morphism (\cite[Theorem 4.2.3 (ii)]{K-S})
\begin{align}
 \RG(U;F) \longrightarrow
\RG(V;\nu_M(F)),
\end{align}
and the restriction, we obtain the one 
\begin{align}
\RG(B \setminus Z;F) \longrightarrow
\RG(\overset{\circ}{\ME_x};\nu_M(F)|_{\ME_x}). \label{mor:6-2}
\end{align}
Now consider the following commutative diagram:
\[
\scalebox{0.95}{
\xymatrix{
\cdots \ar[r] &
\displaystyle\varinjlim_{B,Z}H^{k-1}(B\setminus Z;F) 
\ar[r] \ar[d]_{\gamma_{k-1}}&
\displaystyle\varinjlim_{B,Z}H^k_{B \cap Z}(B;F) 
\ar[r] \ar[d]_{\alpha_k}&
\displaystyle\varinjlim_{B}H^k(B;F) 
\ar[r] \ar[d]_{\beta_k} &
\cdots \\
\cdots \ar[r] &
H^{k-1}(\overset{\circ}{\ME_x};\nu_M(F)|_{\ME_x}) \ar[r] &
H^k_{\{x\}}(\ME_x;\nu_M(F)|_{\ME_x}) \ar[r] &
H^k(\ME_x;\nu_M(F)|_{\ME_x}) \ar[r] & \cdots.
}}
\]
Then all the rows are exact and all $\beta_k$'s are isomorphisms
since $H^k(\ME_x;\nu_M(F)|_{\ME_x}) \simeq H^k(F)_x$.
Thus it suffices to show that $\gamma_k$ is an isomorphism 
for any $k \in \Z$.
By \cite[Theorem 4.2.3 (ii)]{K-S}, we have
\begin{align}
H^{k}(\overset{\circ}{\ME_x};\nu_M(F)|_{\ME_x})
& \simeq
\varinjlim_{W}H^k(W;\nu_M(F)) \label{isom:6-1}\\
& \simeq
\varinjlim_{W,U'}H^k(U';F),\label{isom:6-2}
\end{align}
where $W$ ranges through conic open neighborhoods of $\overset{\circ}{\ME_x}$
in $T_MX$ and $U'$ ranges through open subsets of $X$ such that 
$C_M(X \setminus U') \cap W=\emptyset$.
For a pair $(B,Z)$ as in the statement, 
by taking $U=B \setminus Z$ as $U'$ and 
$V:=\tau^{-1}(B \cap M) \setminus C_M(Z)$ as $W$, we obtain the morphism
\begin{align}
\varinjlim_{B,Z}H^k(B \setminus Z;F)
\longrightarrow
\varinjlim_{W,U'}H^k(U';F). \label{mor:6-3}
\end{align}
Conversely, for any pair $(W,U')$ as above,
if we take any open neighborhood $B$ of $x$ in $X$ and set $Z=X \setminus U'$,
we have $C_M(Z) \cap \overset{\circ}{\ME_x}=\emptyset$ and 
$B \setminus Z \subset U'$.
Hence the morphism (\ref{mor:6-3}) is an isomorphism.
Since the composite of the morphism (\ref{mor:6-3})
and isomorphisms \eqref{isom:6-1}-\eqref{isom:6-2} is equal to $\gamma_k$,
we complete the proof. \qed
\end{proof}

\begin{rem}\label{rem:6-1}
Let $\tau \colon \Vb \longrightarrow M$ be a vector bundle and 
$\psi \colon \Vb \longrightarrow \Vb$ its endomorphism.
Similarly to Definitions \ref{dfn:6-1} and \ref{dfn:6-2}, 
we can define its minimal shrinking subbundle and 
shrinking subbundles (cf.\ \cite[Section 9.6]{K-S}). 
For a shrinking subbundle $\MS$ on an open subset $U \subset M$ 
and a conic $\Real$-constructible object $G \in \Dc(\Vb)$, we set 
\begin{align}
G_{\MS}^{-1!}:=i_\MS^{-1}j_\MS^!G \in \Dc(U),
\end{align}
where $i_\MS \colon U \longhookrightarrow \MS$ is the zero-section 
of $\MS$ and $j_\MS \colon \MS \longhookrightarrow \Vb$ is the 
inclusion map (cf.\ \cite{Braden}).
Moreover as in Definition \ref{dfn:6-3}, 
to a morphism $\Psi \colon \psi^{-1}G \longrightarrow G$ 
we can associate its endomorphism
$\Psi_\MS^{-1!} \colon G_{\MS}^{-1!} \longrightarrow G_{\MS}^{-1!}$.
Then we obtain a result similar to Proposition \ref{prp:6-3} 
for the pair $(G_\MS^{-1!},\Psi_\MS^{-1!})$ 
(cf.\ \cite[Proposition 9.6.14]{K-S}) 
and can define a constructible function 
$\varphi^{\rm s}_M(G,\Psi) \in CF(M)_{\Comp}$ 
globally defined on $M$ and associated to it. 
We can easily show that 
$\varphi^{\rm s}_M(G,\Psi) = \varphi_M(G,\Psi)$. 
Hence we can calculate the value of $\varphi_M(G,\Psi)$
also by shrinking subbundles. 
In fact the proof of the equality
\begin{align}
LC(G,\Psi)=CC(\varphi^{{\rm s}}_M(G,\Psi))
\end{align}
is much easier than that of Proposition \ref{prp:6-2}.
\end{rem}

\section{Some examples}\label{sec:7}

In this section, we introduce various examples to 
which our results in previous sections are applicable. 
First of all, we shall give a very simple proof to the 
following result in \cite[Corollary 6.5]{M-T-3}. 
Note that the original proof in \cite{M-T-3} relies on 
some deep results on the functorial properties of 
Lefschetz cycles (see \cite[Sections 5 and 6]{M-T-3} 
for the details). 
Let $M_i$ be a fixed point component of $\phi$ 
such that $\supp (F) \cap M_i \subset 
(M_i)_{\reg}$. For simplicity, 
we denote $(M_i)_{\reg}$ by $M$. 
Let $M= \sqcup_{\alpha \in A} M_{\alpha}$ be the 
decomposition of 
$M=(M_i)_{\reg}$ into its connected components 
and recall the notations in Introduction. 

\begin{thm}\label{cor:7-11} {\rm \bf(\cite[Corollary 6.5]{M-T-3})} 
In addition to the condition:
\begin{equation}
1 \notin \Ev(\phi^{\prime}_x) \hspace{5mm} 
\text{for any $x \in \supp(F) \cap M$}, 
\end{equation}
assume that the inclusion map $i_M \colon M 
\longhookrightarrow X$ is non-characteristic 
for $F$ and $\supp(F) \cap M$ is compact. Then we have 
\begin{equation}
c(F,\Phi)_M = \sum_{\alpha \in A} \ 
\sgn(\id-\phi^{\prime})_{M_{\alpha}} \cdot 
\dint_{M_{\alpha}} \varphi(F|_{M_{\alpha}},
 \Phi|_{M_{\alpha}}). 
\end{equation}
\end{thm}

\begin{proof}
First, we prove the following lemma.
\begin{lem}\label{lem:7-2}
In the situation of Theorem $\ref{cor:7-11}$, 
$\nu_M(F)|_{\tau^{-1}(x)}$ is smooth, i.e. 
its cohomology sheaves are (locally) constant 
for any $x \in M$.
\end{lem}
\begin{proof}
Set $G=\nu_M(F)$. It suffices to show that 
for any $x \in M$ and $p \in \tau^{-1}(x)$ the restriction morphism
\begin{align}
 \RG(\tau^{-1}(x);G) \longrightarrow G_p 
\end{align}
is an isomorphism. 
For $p=0$ it is trivial by the conicness of $G$. 
Assume that $p \neq 0$.
By the conicness of $G$ we have only to show that 
for any open convex cone
$V \subset T_MX$ containing $p$ and $U=\tau(V)$ the 
restriction morphism 
\begin{align}
 \RG(\tau^{-1}(U);G) \longrightarrow 
\RG(V;G) \label{mor:sprest}
\end{align}
is an isomorphism. 
Let $\pi : T^*_MX \longrightarrow M$ be 
the projection. Then by 
\cite[Proposition 3.7.12]{K-S} and 
the isomorphism 
$G \simeq \mu_M(F)^{\vee}$, the morphism 
(\ref{mor:sprest}) is identified with the one 
\begin{align}
 \RG_U(\pi^{-1}(U);\mu_M(F)) \longrightarrow 
\RG_{V^{\circ a}}(\pi^{-1}(U);\mu_M(F)) 
\label{mor:fourier}
\end{align}
up to some shift. 
Here we identify $U$ with the zero-section 
of $\pi^{-1}(U)$ and $a \colon T^*_MX \longrightarrow T^*_MX$ 
is the antipodal map. 
From the assumption that $i_M \colon M 
\longhookrightarrow X$ is non-characteristic for $F$, we obtain 
\begin{align}
 \supp(\mu_M(F)) \subset \SS(F) \cap T^*_MX 
\subset T^*_XX|_M.
\end{align}
Thus the morphism (\ref{mor:fourier}) is an isomorphism.\qed
\end{proof}
Let us continue the proof of the theorem. 
We may assume that $M=(M_i)_{\reg}$ is connected. 
It suffices to show that for any $x \in M$ we have 
\begin{align}
 \varphi_M(\nu_M(F),\Phi')(x)=
\sgn(\id-\phi^{\prime})_M \cdot 
\varphi(F|_M,\Phi|_M)(x). \label{eq:compcf}
\end{align}
We calculate the left hand side of \eqref{eq:compcf} 
by taking a minimal expanding subbundle 
$\MW \subset T_MX$ 
on a sufficiently small neighborhood of $x$ in $M$. 
Set $\psi:=\phi'|_{\MW_x} \colon \MW_x \longrightarrow \MW_x$ 
and $H=\nu_M(F)|_{\MW_x}$.
Then by Lemma \ref{lem:7-2}, $H$ is 
smooth on $\MW_x$ and hence we have 
\begin{align}
 \RG_c(\MW_x;H) \simeq H_x[-d_x] \simeq 
F_x[-d_x], \label{eq:const}
\end{align}
where we set $d_x=\dim \MW_x$. 
Moreover we define a morphism $\Psi 
\colon \psi^{-1}H \longrightarrow H$ by restricting 
$\Phi' \colon (\phi')^{-1}\nu_M(F) 
\longrightarrow \nu_M(F)$ to $\MW_x$. 
Then the left hand side of \eqref{eq:compcf} is equal to 
the trace of the composite of the morphisms 
\begin{align}
 \RG_c(\MW_x;H) \longrightarrow \RG_c(\MW_x;\psi^{-1}H) 
\overset{\Psi}{\longrightarrow} \RG_c(\MW_x;H).
\end{align}
By the isomorphism (\ref{eq:const}) and $\det\psi>0$, 
it is equal to 
the trace of the composite of the morphisms 
\begin{align}
 F_x[-d_x] \simeq (\phi^{-1}F)_x[-d_x] 
\overset{\Phi_x}{\longrightarrow} F_x[-d_x]. 
 \end{align}
Now the assertion follows immediately from 
the equality $(-1)^{d_x}=\sgn(\id-\phi^{\prime})_M$. 
\qed
\end{proof}

\begin{rem}
Theorem \ref{cor:7-11} is not true if we do not 
assume that $i_M \colon M 
\longhookrightarrow X$ is 
non-characteristic for $F$. See e.g. \cite[Example 
9.6.18]{K-S}.
\end{rem}

We have also a lot of examples as follows.


\begin{exm}
 Let $S^2=\{x=(x_1, x_2, x_3) \in \Real^{3} \ |
\ x_1^2+x_2^2 +x_{3}^2=1\}$ be the
 $2$-dimensional unit sphere in $\Real^{3}$
 and $S^1= \{ e^{i \theta} \ | \ 0 \leq \theta \leq
 2 \pi \}$ the $1$-dimensional one. Set
 $X=S^1 \times S^2$. For $e^{i \theta} \in S^1$
 we define a real analytic
 isomorphism $A_{\theta} \colon \Real^{3} \longrightarrow
\Real^{3}$ of $\Real^{3}$ by
\begin{equation}
 A_{\theta} (x)=\begin{pmatrix}
 2 \cos \theta & -2 \sin \theta & 0\\
2 \sin \theta & 2 \cos \theta & 0\\
0 & 0 & 1
\end{pmatrix}
\begin{pmatrix}x_1\\x_2\\x_3
\end{pmatrix}
\end{equation}
 and the one $\phi \colon X \longrightarrow
 X$ of $X$ by
\begin{equation}
\phi(e^{i \theta}, x)= \left(e^{i \theta},
\dfrac{A_{\theta}(x)}{\|A_{\theta}(x)\|}\right).
\end{equation}
 Then the fixed point set $M$ of $\phi$ is
 a submanifold of $X$ and consists of $3$
 connected components $M_1,M_2, M_3$
 defined by
\begin{equation}
 M_1=S^1 \times (0,0,1), \quad
 M_2=S^1 \times (0,0,-1), \quad
 M_3= \{ 1 \} \times (S^2 \cap \{ x_3=0 \})
\simeq S^{1}
\end{equation}
 respectively. Note that for
 $p=( e^{i \theta}, (0,0,1)) \in M_1$
 the set ${\rm Ev}( \phi^{\prime}_{p})$ of
 the eigenvalues of $\phi^{\prime}_{p} \colon
 (T_{M_1}X)_{p} \longrightarrow (T_{M_1}X)_{p}$
 is given by
 ${\rm Ev}( \phi^{\prime}_{p})= \{
2e^{i \theta}, 2e^{-i \theta} \}$.
In particular, it varies depending on
 the point $p \in M_1$ and satisfies the condition
\begin{equation}
 1 \notin {\rm Ev}( \phi^{\prime}_{p}) \quad
\text{for any} \ p \in M_1.
\end{equation}
 Let $\rho \colon S^2 \setminus
\{ (0,0,1), (0,0,-1) \} \longrightarrow S^1$
 be the natural surjective morphism
 and \\$I_1,I_2, \ldots, I_k \subset S^1$
 closed intervals. Assume that
 $I_1,I_2, \ldots, I_k$ are mutually disjoint 
 and 
\begin{equation}
\exp ( \frac{2 \pi i}{k} ) \cdot 
(I_1 \cup I_2 \cup \cdots \cup I_k)
=(I_1 \cup I_2 \cup \cdots \cup I_k). 
\end{equation}
We denote the closure of
 $\rho^{-1}(I_1 \cup I_2 \cup \cdots \cup I_k)$
 in $S^2$ by $K$. Let us set
\begin{equation}
 Y= \left\{ (e^{i \theta},x) \in 
X \ \Big| \ x_3 > \frac{1}{2} \right\},
\quad
 Z= \{ (e^{i \theta} , x) \in Y \ | \ 
(e^{i \theta})^k=1,  x \in K \}.
\end{equation}
 Then for the constructible sheaf
 $F=\Comp_{Y \setminus Z} \in \Dc(X)$ the inclusion map
 $i_{M_1} \colon M_1 \longhookrightarrow X$ is
 characteristic and there exists a
 natural morphism
 $\Phi : \phi^{-1}F \longrightarrow F$.
 By Theorem \ref{thm:6-4} we have
\begin{equation}
 c(F,\Phi) = c(F,\Phi)_{M_1}= k(k-1).
\end{equation}
 Moreover we can easily see that
\begin{equation}
 {\rm tr}(F,\Phi) = \chi_c (Y)- \chi_c (Z) =0-
k(1-k) =k(k-1).
\end{equation}
Here $\chi_c( \cdot )$ stands for the 
Euler characteristic with compact support. 
Similarly we can construct an example 
for which the set $\Ev(\phi^{\prime}_p)$ 
rotates on a small circle around the 
point $1 \in \Comp$. In this case, we cannot 
take an expanding subbundle of $T_{M_1}X$ 
globally on $M_1$. Such $M_1$ is not weakly 
hyperbolic in the sense of 
Goresky-MacPherson \cite{G-M-1} in general. 
\end{exm}


\begin{exm}
Let ${\rm Hom}_{\rm gr}(\Z^n, \Real \setminus \{0\})$ be the 
abelian group consisting of group homomorphisms of 
the lattice $\Z^n$ to the multiplicative group 
$\Real \setminus \{0\}$. Then the $n$-dimensional real 
algebraic torus $T=(\Real \setminus \{0\})^n$ 
can be naturally identified with 
${\rm Hom}_{\rm gr}(\Z^n,\Real \setminus \{0\})$ 
(see Fulton \cite{Fulton} etc.). 
Let $t \in T= {\rm Hom}_{\rm gr}(\Z^n, \Real \setminus \{0\})$ 
be an element of $T$ satisfying the condition: 
\begin{align}
& \text{There exists a rational linear subspace }
L \subset \Real^n \notag \\
& \text{such that }
\dim L \ge 1 \text{ and }
{\rm Ker}\, t=L \cap \Z^n.
\end{align}
Let $\Sigma$ be a complete smooth fan in $\Real^n$.
Assume that there is a cone $\sigma \in \Sigma$ such that
\begin{align}
\dim \sigma <n \text{ and } \sigma^{\perp} \subset \Real \otimes {\rm Ker}\,t.
\end{align}
Let $X:=X_{\Sigma}$ be the complete smooth real toric variety 
associated with $\Sigma$ and 
$\phi :=l_t \colon X \longrightarrow X$
the natural action of $t$ on it. 
Then the fixed point set $M$ of $\phi$ is explicitly given by
\begin{align}
M=
\bigsqcup_{\sigma^{\perp} \subset 
\Real \otimes {\rm Ker}\,t} T_\sigma,
\end{align}
where $T_{\sigma} \simeq (\Real \setminus \{0\})^{n- \dim \sigma}$ 
is the $T$-orbit associated with the cone $\sigma \in \Sigma$.
Define a partial order $\preceq$ on $\Sigma$ by 
\begin{align}
\sigma \preceq \tau \ \Longleftrightarrow \ 
\text{$\sigma$ is a face of $\tau$}.
\end{align}
Let $\sigma_1,\dots,\sigma_k$ be minimal elements of the set
$\{\sigma \in \Sigma \ | \ \sigma^{\perp} \subset 
\Real \otimes {\rm Ker}\,t\}$ with respect to the order $\preceq$.
Set $M_i:=\overline{T_{\sigma_i}} \subset X$.
Then $M_i$ is a connected submanifold of $X$ 
such that $\dim M_i=n-\dim \sigma_i \ge 0$. 
Moreover $M=\bigsqcup_{i=1}^{k}M_i$ is the decomposition 
of $M$ into connected components and we have 
$\dim M \geq 1$. Let $F \in \Dc(X)$ 
be an object satisfying the condition
\begin{align}
H^j(F|_{T_{\sigma}})
\text{ is a constant sheaf for any }
\sigma \in \Sigma \text{ and } 
j \in \Z. \label{cd:7-1}
\end{align}
and $\Phi \colon \phi^{-1}F \longrightarrow F$
a morphism in $\Dc(X)$.
Let us apply our fixed point formula 
to the pair $(F,\Phi)$. For 
the fixed point component $M_i=\overline{T_{\sigma_i}}$ 
associated with a minimal cone 
$\sigma_i$ we can compute the constructible function 
$\theta(F,\Phi)_{M_i}=\varphi_{M_i}(\nu_{M_i}(F),\Phi')$ 
as follows. Set $d:=\dim M_i$ and choose 
an $n$-dimensional cone $\tau_i \in \Sigma$ such that 
$\sigma_i \preceq \tau_i$. Then 
$U_{\tau_i}:={\rm Hom}_{\rm gr}(\tau_i^{\vee} \cap \Z^n, \Real) 
\simeq \Real^n$ is an affine open subset of $X$ 
containing the $T$-orbit $T_{\sigma_i}$. On 
$U_{\tau_i} \simeq \Real^n$ there exists a 
coordinate $(x_1,x_2, \ldots, x_n)$ such that 
$M_i \cap U_{\tau_i}=\{x \in \Real^n \ | 
\ x_{d+1}=\dots=x_n=0 \}$ and the map $\phi=l_t$ 
can be explicitly written as 
\begin{align}
(x_1,\dots,x_d,x_{d+1},\dots,x_n) \longmapsto 
(x_1,\dots,x_d,t_{d+1} x_{d+1},\dots,t_n x_n) 
\end{align}
for some $t_{d+1},\dots,t_n \neq 0,1$. 
Then we can identify $T_{M_i}(X)$ 
with $\Real^n$ on $U_{\tau_i}$.
Furthermore, by the condition \eqref{cd:7-1}, we can identify 
the pair $(\nu_{M_i}(F),\Phi')$ with $(F,\Phi)$.
Interchanging the coordinates, we may assume also that
$t_{d+1},\dots,t_{d+m}>1, t_{d+m+1},\dots,t_n<1$.
Then we can take the subbundle 
\begin{align}
\ME=
\{x \in \Real^n \ | \ x_{d+m+1}=\dots=x_n=0 \}
\simeq (M_i \cap U_{\tau_i}) \times \Real^m
\end{align}
as an expanding subbundle of $T_{M_i}(X)$ 
on $M_i \cap U_{\tau_i}$. 
Thus, the value of the function 
$\theta:=\theta(F,\Phi)_{M_i}$ 
at a point $x \in M_i \cap U_{\tau_i}$ is given by
\begin{align}
\theta(x)&=
\tr\{\RG_{\{x\}}(\ME_x;F|_{\ME_x})\} \\
&=\tr(F_x)-\tr\{\RG(\ME_x \setminus \{x\};F)\},
\end{align}
where we denote by $\tr\{\RG_{\{x\}}(\ME_x;F|_{\ME_x})\}$ etc.\
the trace of the induced endomorphism of 
$\RG_{\{x\}}(\ME_x;F|_{\ME_x})$ etc. 
Note that we have 
$\RG(\ME_x;F|_{\ME_x}) \simeq F_x$ 
by the conicness of $F|_{\ME_x}$.
We set $S(\ME_x):=(\ME_x \setminus \{x\})/\Real_{>0} 
\simeq S^{m-1}$ 
and identify it with the unit sphere of $\ME_x$. 
Let $\gamma \colon \ME_x \setminus \{x\} 
\longrightarrow S(\ME_x)$ 
the the natural map. 
Then by the conicness of $F|_{\ME_x}$, we have 
\begin{align}
\theta(x)=
\tr(F_x)-\tr(R\gamma_*(F|_{\ME_x \setminus \{ x \} }),
R\gamma_*(\Phi|_{\ME_x \setminus \{ x \} })),
\end{align}
where $R\gamma_*(\Phi|_{\ME_x \setminus \{ x \}})$ 
is a natural lift of the map 
$\gamma_*(\phi|_{\ME_x \setminus \{x\}}) \colon S(\ME_x) 
\longrightarrow S(\ME_x)$. 
In order to give a more explicit 
description of the value $\theta(x)$, 
assume the following condition:
\begin{align}
t_{d+1}\dots,t_{d+m} \text{ are distinct}.
\end{align}
We may assume that $t_{d+1}>\dots>t_{d+m}>1$ 
Denote by $\xi_1,\dots,\xi_m$ the corresponding 
coordinates of $\ME_x\simeq \Real^m$. 
Then the fixed point set of $\gamma_*(\phi|_{\ME_x \setminus \{x\}})$ is 
the intersection of the unit sphere with the $\xi_i$-axes. 
For $1 \leq i \leq m$ let us calculate 
the local contribution of the pair 
$(R\gamma_*(F|_{\ME_x \setminus \{ x \} }),
R\gamma_*(\Phi|_{\ME_x \setminus \{ x \} }))$ 
at $p=(0,\dots,0,\stackrel{\overset{i}{\smile}}{\pm 1},
0,\dots,0) \in S(\ME_x)$. 
The tangent space 
$V(p):=T_pS(\ME_x)$ can be identified 
with the $(m-1)$-dimensional 
linear subspace of $\Real^m$ spanned by 
$\xi_1 \dots \xi_{i-1},\xi_{i+1},\dots,\xi_m$ and realized in $\ME_x$ as  
\begin{align}
 V(p)=\{(\xi_1,\dots,\xi_{i-1},\pm 1,\xi_{i+1},\dots,\xi_m) \ | \ \xi_j \in 
 \Real \}.
\end{align}
Moreover by the condition \eqref{cd:7-1}, we can identify 
$\nu_p(R\gamma_*(F|_{\ME_x \setminus \{x\}}))$ with $F|_{V(p)}$.
Under these identifications, the map on $V(p)$ induced by 
$\gamma_* (\phi|_{\ME_x \setminus \{x\}})$ can be explicitly written as 
\begin{equation}
 (\xi_1,\dots,\xi_{i-1},\xi_{i+1},\dots,\xi_m) \longmapsto 
 (u_1\xi_1,\dots,u_{i-1}\xi_{i-1},u_{i+1}\xi_{i+1},\dots,u_m\xi_m) 
\end{equation}
for some $u_1>\dots>u_{i-1}>1>u_{i+1}>\dots>u_m>0$. 
Thus at the point $p$, we can 
take the $(i-1)$-dimensional subspace of $T_pS(\ME_x)$ 
spanned by $\xi_1 \dots \xi_{i-1}$ as an expanding subbundle (subspace)
and this is realized as
\begin{align}
 W(p)=\{(\xi_1,\dots,\xi_{i-1},\pm 1,0,\dots,0) \ | \ \xi_j \in \Real\}.
\end{align}
Therefore we have 
\begin{align}
c(R\gamma_*(F|_{\ME_x \setminus \{x\}}),R\gamma_*(\Phi|_{\ME_x \setminus \{x\}}))_p
&=\tr \{ \RG_{\{p\}}(W(p);F|_{W(p)}) \} \\
&=\tr(F_p)-
\tr \{ \RG(W(p) \setminus \{p\};F) \}. 
\end{align}
Repeating this argument, 
we can easily prove the following very simple formula: 
\begin{align}
\theta(x)=
\sum_{J \subset \{1,\dots,m\}}
\sum_{\stackrel{\e=(\e_j)_{j \in J}}{\e_j \in \{\pm 1\}}}
(-1)^{|J|} \tr\{\RG(T_{J,\e};F|_{T_{J,\e}})\},
\end{align}
where $|J|$ denotes the cardinality of $J$.
Here, for a multi-sign $\e=(\e_j)_{j \in J},\ \e_j \in \{\pm 1\}$ indexed by $J$, 
we set 
\begin{align}
T_{J,\e}=\{q \in \ME_x \simeq \Real^m \ | \
\e_j q_j >0 \ (j \in J),\
q_j=0 \ (j \in \{1,\dots,m\} \setminus J) \}.
\end{align}
\end{exm}


\end{document}